\documentclass[11pt]{article}

\usepackage[a4paper,margin=1in]{geometry}
\usepackage[utf8]{inputenc}
\usepackage{microtype}
\usepackage{amsmath,amssymb,amsthm,mathtools}
\usepackage{bm}
\usepackage{booktabs}
\usepackage{enumitem}
\usepackage[numbers,sort&compress]{natbib}
\usepackage[colorlinks=true,linkcolor=blue,citecolor=blue,urlcolor=blue]{hyperref}
\usepackage[nameinlink,noabbrev]{cleveref}

\allowdisplaybreaks
\numberwithin{equation}{section}

\newtheorem{theorem}{Theorem}[section]
\newtheorem{proposition}[theorem]{Proposition}
\newtheorem{lemma}[theorem]{Lemma}

\newtheorem{assumption}[theorem]{Assumption}
\theoremstyle{definition}

\newtheorem{example}[theorem]{Example}
\theoremstyle{remark}
\newtheorem{remark}[theorem]{Remark}
\crefname{assumption}{Assumption}{Assumptions}
\Crefname{assumption}{Assumption}{Assumptions}

\newcommand{\R}{\mathbb{R}}

\newcommand{\Z}{\mathbb{Z}}
\newcommand{\E}{\mathbb{E}}
\newcommand{\Pp}{\mathbb{P}}
\newcommand{\1}{\mathbf{1}}
\newcommand{\dd}{\,\mathrm{d}}
\newcommand{\floor}[1]{\lfloor #1\rfloor}

\newcommand{\SV}{\mathrm{SV}}
\newcommand{\fdd}{\mathrm{f.d.d.}}
\newcommand{\RightarrowM}{\Rightarrow_{M_1}}
\newcommand{\RightarrowJ}{\Rightarrow_{J_1}}
\newcommand{\e}{\mathrm{e}}

\title{\textbf{Limit Theorems for Tempered Linear Processes with Innovations in the Domain of Attraction of a Stable Law}}
	\author{
	Qian Yu\thanks{Corresponding author, Email: qyumath@163.com}\\
		{\footnotesize\itshape School of Mathematics, Nanjing University of Aeronautics and Astronautics}\\
		{\footnotesize\itshape Nanjing, Jiangsu 211106, P.R. China}
	}
\date{\today}

\begin{document}
\maketitle

\begin{abstract}
We study the partial-sum behavior of  tempered linear processes
\[
X_{N,n}=\sum_{j=1}^{\infty}e^{-\lambda_Nj}\frac{\ell(j)}{j}\varepsilon_{n-j},
\qquad \lambda_N\downarrow0,
\]
where $\ell$ is slowly varying and the innovations belong to the domain of attraction of an $\alpha$-stable law with $1<\alpha\leq2$. The filter $j^{-1}\ell(j)$ represents the logarithmic boundary between summable and power-law long-memory coefficients. Let
\[
Q_N=\sum_{j=1}^{N}e^{-\lambda_Nj}\frac{\ell(j)}{j},
\qquad
L(N)=\sum_{j=1}^{N}\frac{\ell(j)}{j}.
\]
We prove that the partial-sum process, normalized by $B_NQ_N$, converges to the stable L\'{e}vy motion associated with the innovations. Moreover,
\[
Q_N\sim L(N)\quad\text{if }N\lambda_N=O(1),
\qquad
Q_N\sim L(1/\lambda_N)\quad\text{if }N\lambda_N\to\infty.
\]
Then weak, moderate, and strong tempering have the same first-order L\'{e}vy limit but different normalizations. 
In the weakly and moderately tempered regimes, the second-order remainder, normalized by $B_N\ell(N)$, converges in finite-dimensional
distributions to a logarithmically tempered stable process; in the Gaussian finite-moment case the convergence is functional.
These results extend the untempered logarithmic-boundary theorem and complement existing invariance principles for tempered linear processes.
\end{abstract}



\section{Introduction}
\label{sec:introduction}

Linear processes are among the basic models in probability and time
series analysis.  Given an i.i.d.\ innovation sequence
$\{\varepsilon_k:k\in\Z\}$ and a deterministic filter $\{a_j\}$, one
considers
\[
 X_n=\sum_{j\geq0}a_j\varepsilon_{n-j}.
\]
The asymptotic behavior of $\sum_{n\le Nt}X_n$ is governed jointly by
the tail of the innovations and by the accumulation of the filter.
For absolutely summable coefficients, stable L\'evy limits go back to
\citet{Astrauskas1983} and \citet{DavisResnick1985}; functional versions
and the role of the Skorokhod topologies were developed in
\citet{AvramTaqqu1992}, \citet{BalanJakubowskiLouhichi2016}, and the
subsequent point-process literature.  In particular, several adjacent
jumps produced by one large innovation cannot in general be matched in
the $J_1$ topology, whereas a same-sign cluster can be compressed to
one jump in the $M_1$ topology.

If $a_j$ decreases like $j^{d-1}$ with $d>0$, the cumulative filter has
order $N^d$ and the long-memory effect survives in the limit, usually
in the form of a linear fractional stable motion.  At the opposite
end, a summable filter produces a stable L\'evy motion.  The
coefficient
\[
 a_j=\frac{\ell(j)}{j},\qquad \ell\in\SV_\infty,
\]
is the critical logarithmic boundary. $\SV_\infty$ represents the set of slowly varying functions at infinity.  The cumulative sum
\[
 L(x):=\sum_{1\le j\le x}\frac{\ell(j)}{j}
\]
is slowly varying and can diverge even though
$\sum_j |a_j|^\alpha<\infty$ for every $\alpha>1$.  \citet{Xu2025}
proved, under a general stable domain-of-attraction assumption, that
the untempered partial sums normalized by $B_NL(N)$ converge in finite-dimensional distributions to a stable L\'evy motion.  A striking
feature of that result is that the critical memory changes the scale
but not the independent-increment structure of the limit.

Tempering provides a natural way to interpolate between power-law
memory at observable scales and exponential decay at large scales.
Tempered fractional Brownian and stable motions were introduced and
studied in \citet{MeerschaertSabzikar2013,
MeerschaertSabzikar2016}; see also
\citet{SabzikarSurgailis2018b}.  For a power-law filter
$b_d(j)\sim c_dj^{d-1}/\Gamma(d)$, \citet{SabzikarSurgailis2018a}
considered
\[
 X_{d,\lambda_N}(n)
 =\sum_{j=0}^{\infty}\e^{-\lambda_Nj}b_d(j)\varepsilon_{n-j}
\]
and established different invariance principles according as
$N\lambda_N\to0$ (Weak tempering), $N\lambda_N\to\lambda_*\in(0,\infty)$ (moderate tempering), or
$N\lambda_N\to\infty$ (strong tempering).  The moderate regime leads to a tempered
fractional stable motion of the second kind.  Their boundary case
$d=0$, however, is defined separately by the short-memory condition
$\sum_j|b_0(j)|<\infty$.  It therefore does not cover the non-summable
critical filter $\ell(j)/j$.

The purpose of this paper is to fill that boundary gap.  The critical
case behaves differently from a direct extrapolation of the
$d\ne0$ theory.  At first order, the tempering parameter is absorbed by
the slowly varying normalization, and even moderate tempering produces
the same L\'evy limit as no tempering.  The tempering parameter becomes
visible only at the smaller scale $B_N\ell(N)$.  This yields a
logarithmically tempered second-order process.

The terminology ``semi-long memory'' is literal in the finite-variance
submodel.  For example, if $\ell\equiv1$ and
$\operatorname{Var}(\varepsilon_0)=\sigma_\varepsilon^2<\infty$, then
for fixed $\lambda>0$,
\[
 \gamma_\lambda(k)
 =\sigma_\varepsilon^2\e^{-\lambda k}
   \sum_{j=1}^{\infty}
   \frac{\e^{-2\lambda j}}{j(j+k)}.
\]

As $\lambda\downarrow0$ and $k\to\infty$,
\[
 \gamma_\lambda(k)
 \sim
 \sigma_\varepsilon^2\frac{\log k}{k},
 \qquad k\lambda\to0,
\]
whereas
\[
 \gamma_\lambda(k)
 \sim
 \sigma_\varepsilon^2
 e^{-\lambda k}\frac{\log(1/\lambda)}{k},
 \qquad k\lambda\to\infty.
\]
Thus the process looks
long-memory over a growing intermediate range but eventually has an
exponential cutoff.  In the infinite-variance setting, covariance must
be replaced by codifference, covariation, or a distributional
partial-sum notion of memory; see
\citet{SamorodnitskyTaqqu1994} and
\citet{GiraitisKoulSurgailis2012}.

We use a formulation that includes both the non-Gaussian stable and the
Gaussian domain of attraction.

\begin{assumption}[Stable domain of attraction]
\label{ass:innovations}
The random variables $\{\varepsilon_j:j\in\Z\}$ are i.i.d.\
and $\E\varepsilon_0=0$.

\begin{enumerate}[label=(\roman*)]
\item If $1<\alpha<2$, there are $p_+,p_-\ge0$,
$p_++p_-=1$, and a slowly varying function $h$ such that
\[
 \Pp(\varepsilon_0>x)\sim p_+x^{-\alpha}h(x),\qquad
 \Pp(\varepsilon_0<-x)\sim p_-x^{-\alpha}h(x).
\]
Choose $B_N$ so that
\[
 N B_N^{-\alpha}h(B_N)\longrightarrow1, ~~~~\text{as}~~N\to \infty.
\]
\item If $\alpha=2$, the truncated variance
\[
 V(x)=\E\!\left[\varepsilon_0^2
             \1_{\{|\varepsilon_0|\le x\}}\right]
\]
is slowly varying, and $B_N$ is chosen so that
\[
 N B_N^{-2}V(B_N)\longrightarrow1, ~~~~\text{as}~~N\to \infty.
\]
\end{enumerate}
\end{assumption}

Under \cref{ass:innovations},
\begin{equation}
 U_N(t):=\frac1{B_N}\sum_{j=1}^{\floor{Nt}}\varepsilon_j
 \ \RightarrowJ\ Z_\alpha(t)
 \quad\text{in }D[0,T],
 \label{eq:iid-fclt}
\end{equation}
as $N$ tends to infinity, where $Z_\alpha$ is the stable L\'evy motion determined by
$(p_+,p_-)$ when $\alpha<2$, and is a Brownian motion when
$\alpha=2$.  This normalization fixes the scale of the limiting stable law
without introducing an explicit scale constant. It is
equivalent to the normalization
$N^{1/\alpha}H_\alpha(N)^{1/\alpha}$ used by
\citet{Xu2025}, up to asymptotic equivalence.

\begin{assumption}[Critical filter]
\label{ass:filter}
The function $\ell:[1,\infty)\to(0,\infty)$ is measurable, locally
bounded, and slowly varying at infinity.  Set
\[
 a_j=\frac{\ell(j)}{j},\qquad
 L(x)=\sum_{1\le j\le x}a_j.
\]
We assume the genuinely critical, non-summable condition
\begin{equation}
 L(x)\longrightarrow\infty, ~~~~\text{as} ~~x\to\infty.
 \label{eq:Ldiv}
\end{equation}
\end{assumption}

No monotonicity of $\ell$ is required.  Positivity is important for
the $M_1$ functional theorem because it forces the jumps in a cluster
generated by one innovation to have the same sign.

\begin{assumption}[Vanishing tempering]
\label{ass:tempering}
The numbers $\lambda_N$ are positive, as $N\to\infty$, $\lambda_N\to0$, and
\[
 N\lambda_N\longrightarrow\lambda_*
 \in[0,\infty].
 \]
\end{assumption}

For each $N$, define the stationary tempered linear process
\[
 X_{N,n}
 :=\sum_{j=1}^{\infty}b_{N,j}\varepsilon_{n-j},
 \qquad
 b_{N,j}:=\e^{-\lambda_Nj}a_j.
 \]
Since $\lambda_N>0$, the filter is absolutely summable for each fixed
$N$ and the series converges in $L^1$ and almost surely.  Define
\[
 Q_N(x):=\sum_{j=1}^{\floor{Nx}}b_{N,j},
 \qquad Q_N:=Q_N(1),
 \]
and the partial-sum process
\begin{equation}
 S_N(t):=\sum_{n=1}^{\floor{Nt}}X_{N,n},
 \qquad t\ge0.
 \label{eq:SN}
\end{equation}

The following result gives the exact normalization in all three
tempering regimes.

\begin{theorem}[Tempered critical scale]
\label{thm:scale}
Under \cref{ass:filter,ass:tempering}, for every $c>0$,
\[
 \frac{Q_N(c)}{Q_N}\longrightarrow1, ~~~~\text{as}~~N\to \infty.
 \]
Moreover,
\[
 Q_N\sim
 \begin{cases}
 L(N),&0\le\lambda_*<\infty,\\[1mm]
 L(1/\lambda_N),&\lambda_*=\infty.
 \end{cases}
 \]
Equivalently,
\[
 Q_N\sim L\!\left(N\wedge\lambda_N^{-1}\right).
\]
\end{theorem}

\begin{theorem}[First-order functional limit]
\label{thm:firstorder}
Let \cref{ass:innovations,ass:filter,ass:tempering} hold and let
$T<\infty$.  Then
\begin{equation}
 \left\{\frac{S_N(t)}{B_NQ_N}:0\le t\le T\right\}
 \xrightarrow{\fdd}
 \{Z_\alpha(t):0\le t\le T\},
 \label{eq:firstfdd}
\end{equation}
as $N$ goes to infinity.
If $1<\alpha<2$, then the convergence is functional:
\begin{equation}
 \frac{S_N}{B_NQ_N}\RightarrowM Z_\alpha
 \quad\text{in }D[0,T].
 \label{eq:firstM1}
\end{equation}
If $\alpha=2$, $\E\varepsilon_0^2<\infty$, and
$\E|\varepsilon_0|^{2+\delta}<\infty$ for some $\delta>0$, then
\begin{equation}
 \frac{S_N}{B_NQ_N}\RightarrowJ Z_2
 \quad\text{in }D[0,T].
 \label{eq:firstJ1}
\end{equation}
Because the limit in \eqref{eq:firstJ1} is continuous, this is
equivalent to locally uniform weak convergence after polygonal
interpolation.
\end{theorem}

\begin{remark}[Why $M_1$ and not $J_1$]
\label{rem:J1failure}
Under \eqref{eq:Ldiv},
\[
 \max_{j\ge1}\frac{b_{N,j}}{Q_N}\longrightarrow0,~~~~\text{as}~~N\to \infty.
\]
Consequently, one large innovation is transformed into a staircase of
many individually negligible jumps whose total height is asymptotically
the innovation itself.  Such a staircase can converge to a single jump
in $M_1$, but not in $J_1$.  For a non-Gaussian stable limit,
$J_1$ convergence cannot in general be expected, compare \citet{AvramTaqqu1992}.
\end{remark}

\begin{remark}[Weak, moderate, and strong tempering]
\label{rem:three-regimes}
The regimes have the following interpretation.
\begin{enumerate}[label=(\roman*)]
\item If $N\lambda_N\to0$, the cutoff length
$\lambda_N^{-1}$ is much larger than the observation horizon.  The
tempering is asymptotically invisible and $Q_N\sim L(N)$.
\item If $N\lambda_N\to\lambda_*\in(0,\infty)$, the cutoff is of order
$N$.  For a noncritical power-law filter this regime produces a
nontrivial tempered fractional limit.  At the critical index, however,
the mass in every fixed proportional interval $[cN,N]$ is only of
order $\ell(N)=o(L(N))$.  Hence $Q_N\sim L(N)$ and the first-order
limit is again $Z_\alpha$.
\item If $N\lambda_N\to\infty$, the cutoff length is $o(N)$.
The process is distributionally short-memory at the sample scale and
$Q_N\sim L(1/\lambda_N)$.
\end{enumerate}
\end{remark}

\begin{example}[The logarithmic filter]
\label{ex:logfilter}
If $\ell\equiv1$, then $L(N)\sim\log N$.  For
$\lambda_N=N^{-\gamma}$,
\[
 Q_N\sim
 \begin{cases}
 \log N,&\gamma\ge1,\\
 \gamma\log N,&0<\gamma<1.
 \end{cases}
\]
For moderate tempering $\lambda_N=\lambda_*/N$,
\[
 Q_N
 =\log N+\gamma_{\!E}
   +\int_0^1\frac{\e^{-\lambda_*u}-1}{u}\dd u+o(1),
\]
where $\gamma_{\!E}$ is Euler's constant.  Thus $\lambda_*$ first
appears one order below the leading $\log N$ normalization.
\end{example}

The preceding example suggests subtracting the leading innovation
partial sum and magnifying the remainder by $B_N\ell(N)$.  This
requires weak or moderate tempering; throughout the second-order
result we assume
\begin{equation}
 N\lambda_N\longrightarrow\lambda\in[0,\infty).
 \label{eq:secondregime}
\end{equation}
For $x>0$, set
\[
 G_\lambda(x)
 :=\log x+\int_0^x\frac{\e^{-\lambda u}-1}{u}\dd u,
\]
and, for $t\ge0$, $y\in\R$, define
\begin{equation}
 K_\lambda(t,y)
 :=G_\lambda(t-y)\1_{\{y<t\}}
   -G_\lambda(-y)\1_{\{y<0\}}.
 \label{eq:Kkernel}
\end{equation}
The apparent logarithmic singularities are locally integrable.  In
fact,
\[
 K_\lambda(t,\cdot)\in L^r(\R)
 \qquad\text{for every }r>1.
\]
Let $M_\alpha$ be the independently scattered stable random measure
whose cumulative process is $Z_\alpha$, and define
\[
 \mathcal L_{\alpha,\lambda}(t)
 :=\int_{\R}K_\lambda(t,y)\,M_\alpha(\dd y).
\]

\begin{theorem}[Second-order logarithmically tempered limit]
\label{thm:secondorder}
Let \cref{ass:innovations,ass:filter} hold and assume
\eqref{eq:secondregime}.  Put
\begin{equation}
 R_N(t):=
 \frac{1}{B_N\ell(N)}
 \left\{
 S_N(t)-L(N)\sum_{j=1}^{\floor{Nt}}\varepsilon_j
 \right\}.
 \label{eq:RN}
\end{equation}
Then as $N$ tends to infinity,
\begin{equation}
 \{R_N(t):t\ge0\}
 \xrightarrow{\fdd}
 \{\mathcal L_{\alpha,\lambda}(t):t\ge0\}.
 \label{eq:secondfdd}
\end{equation}
If $\alpha=2$, $\E|\varepsilon_0|^{2+\delta}<\infty$ for some
$\delta>0$, and the step processes are polygonally interpolated, then
for every $T<\infty$,
\begin{equation}
 R_N\Rightarrow\mathcal L_{2,\lambda}
 \quad\text{in }C[0,T].
 \label{eq:secondC}
\end{equation}
\end{theorem}

\begin{remark}[Optimality for non-Gaussian limits]
\label{rem:secondpaths}
For $1<\alpha<2$, the logarithmic singularity at $y=t$ is encountered
by the atoms of the stable random measure.  The process
$\mathcal L_{\alpha,\lambda}$ generally has no c\`adl\`ag
modification.  Therefore \eqref{eq:secondfdd} should not be promoted to
$D[0,T]$ convergence without smoothing the process or changing the
state space. 
\end{remark}

\begin{remark}[Connection with TFSM II]
\label{rem:derivative}
Let $Z^{\mathrm{II}}_{H,\alpha,\lambda}$ denote the tempered
fractional stable motion of the second kind (see eq.(1.8) in \cite{SabzikarSurgailis2018a}):
$$
Z^{\mathrm{II}}_{H,\alpha,\lambda}=\int_{\mathbb{R}}h_{H,\alpha,\lambda}(t;y)M_\alpha(\dd y), ~~~t\in\mathbb{R},
$$
with
\begin{align*}
h_{H,\alpha,\lambda}(t;y)=(t-y)_+^{H-1/\alpha}e^{-\lambda(t-y)_+}-(-y)_+^{H-1/\alpha}e^{-\lambda(-y)_+}
+\lambda\int_0^t(s-y)_+^{H-1/\alpha}e^{-\lambda(s-y)_+}\dd s.
\end{align*}
At the critical value
$H=1/\alpha$, its kernel reduces to
\[
 h_{1/\alpha,\alpha,\lambda}(t;y)=\1_{(0,t)}(y),
\]
and hence
$Z^{\mathrm{II}}_{1/\alpha,\alpha,\lambda}=Z_\alpha$ for every
$\lambda\ge0$.  Differentiating its deterministic kernel with respect
to $d=H-1/\alpha$ at $d=0$ gives \eqref{eq:Kkernel}.  More precisely,
at the level of stable integrals,
\[
 \mathcal L_{\alpha,\lambda}(t)
 =
 \left.\frac{\partial}{\partial d}
 \left\{\Gamma(d+1)^{-1}
 Z^{\mathrm{II}}_{d+1/\alpha,\alpha,\lambda}(t)\right\}
 \right|_{d=0}
 -\gamma_{\!E}Z_\alpha(t).
\]
Thus the second-order limit is the logarithmic tangent process to the
TFSM II family at its L\'evy boundary.
\end{remark}

The contributions are fourfold.  First, we identify the non-summable
$d=0$ boundary class omitted by the standard tempered fractional
theory.  Second, we obtain a unified coefficient normalizer and a
complete weak--moderate--strong phase diagram.  Third, we strengthen
the boundary finite-dimensional theorem to an $M_1$ functional limit
under the natural positivity assumption.  Fourth, we identify a new
second-order logarithmically tempered stable limit that retains the
moderate tempering parameter erased at first order.

The rest of the paper has three sections.  \Cref{sec:preliminaries}
collects regular-variation, remote-past, cluster-compression, and
weighted-sum lemmas.  \Cref{sec:firstproof} proves the first-order
finite-dimensional and functional limits.  \Cref{sec:secondproof}
proves the second-order theorem and its Gaussian functional version.

\section{Preliminaries and technical lemmas}
\label{sec:preliminaries}

\subsection{Slow variation at the logarithmic boundary}

We repeatedly use the uniform convergence theorem and Potter bounds
for slowly varying functions; see \citet{BinghamGoldieTeugels1987} or
\citet{Resnick2007}.  Sums and integrals are interchangeable at the
precision used below.

\begin{lemma}[Integrated slow variation]
\label{lem:integratedSV}
Under \cref{ass:filter}, $L$ is slowly varying,
\begin{equation}
 \frac{\ell(x)}{L(x)}\longrightarrow0, ~~~~\text{as} ~~x\to\infty,
 \label{eq:elloverL}
\end{equation}
and, for every $c>0$,
\begin{equation}
 \frac{L(cx)-L(x)}{\ell(x)}\longrightarrow\log c.
 \label{eq:deHaan}
\end{equation}
More generally, if $0<a<b<\infty$ and $f$ is continuous on $[a,b]$,
then as $N\to\infty$,
\begin{equation}
 \frac1{\ell(N)}
 \sum_{aN<j\le bN}\frac{\ell(j)}{j}f(j/N)
 \longrightarrow
 \int_a^b\frac{f(u)}{u}\dd u.
 \label{eq:RiemannSV}
\end{equation}
The same conclusion holds with $a=0$ whenever
$f(u)=O(u^\rho)$ as $u\downarrow0$ for some $\rho>0$.
\end{lemma}

\begin{proof}
For $c>1$,
\[
L(cx) - L(x) = \sum_{x < j \le cx} \frac{\ell(j)}{j}.
\]
Then
\[
\frac{L(cx) - L(x)}{\ell(x)}
= \sum_{x < j \le cx} \frac{1}{j}\,\frac{\ell(j)}{\ell(x)}.
\]
Since $\ell$ is a slowly varying function, the uniform convergence theorem for slowly varying functions yields
\[
\sup_{1\le u\le c} \left| \frac{\ell(xu)}{\ell(x)} - 1 \right| \to 0, ~~~x\to\infty.
\]
In particular,
\[
\sup_{x<j\le cx} \left| \frac{\ell(j)}{\ell(x)} - 1 \right| \to 0.
\]
Hence we split the sum:
\[
\frac{L(cx) - L(x)}{\ell(x)}
= \sum_{x<j\le cx} \frac{1}{j}
+ \sum_{x<j\le cx} \frac{1}{j}\left( \frac{\ell(j)}{\ell(x)} - 1 \right).
\]
For the second term, we have the bound
\[
\left| \sum_{x<j\le cx} \frac{1}{j}\left( \frac{\ell(j)}{\ell(x)} - 1 \right) \right|
\le \sup_{x<j\le cx}\left| \frac{\ell(j)}{\ell(x)} - 1 \right| \cdot \sum_{x<j\le cx}\frac{1}{j}
\to 0.
\]
On the other hand, the harmonic sum converges to the integral:
\[
\sum_{x<j\le cx}\frac{1}{j} \;\to\; \int_{1}^{c}\frac{du}{u} = \log c.
\]
Combining these results, we conclude
\[
\frac{L(cx) - L(x)}{\ell(x)} \to \log c.
\]

The case $c<1$ follows by reversing the sum
\[
L(cx) - L(x) = -\sum_{cx < j \le x} \frac{\ell(j)}{j}.
\]
The statement holds trivially when $c=1$. Hence for all $c>0$, as $x\to\infty$,
\[
\frac{L(cx) - L(x)}{\ell(x)} \to \log c.
\]

Relation \eqref{eq:RiemannSV} is the corresponding Riemann-sum statement,
uniformly on compact subsets of $(0,\infty)$.  Let $0 < a < b < \infty$, and suppose $f$ is continuous on $[a,b]$. Define
\[
g(u) := \frac{f(u)}{u}.
\]
Since $a>0$, the function $g$ is continuous and bounded on $[a,b]$.
Rewrite the sum under consideration as
\[
\begin{aligned}
S_N(a,b) := \frac{1}{\ell(N)} \sum_{aN < j \le bN} \frac{\ell(j)}{j} f\!\left(\frac{j}{N}\right) 
= \frac{1}{N} \sum_{aN < j \le bN} \frac{\ell(j)}{\ell(N)} g\!\left(\frac{j}{N}\right).
\end{aligned}
\]
By the uniform convergence theorem for slowly varying functions, set
\[
\Delta_N := \sup_{aN < j \le bN} \left| \frac{\ell(j)}{\ell(N)} - 1 \right| \to 0.
\]
Therefore,
\[
\left| S_N(a,b) - \frac{1}{N}\sum_{aN < j \le bN} g\!\left(\frac{j}{N}\right) \right|
\le \Delta_N \cdot \frac{1}{N}\sum_{aN < j \le bN} \left|g\!\left(\frac{j}{N}\right)\right|
\to 0.
\]
On the other hand, convergence of the standard Riemann sum yields
\[
\frac{1}{N}\sum_{aN < j \le bN} g\!\left(\frac{j}{N}\right) \to \int_{a}^{b} g(u)\,\mathrm{d}u.
\]
Hence we obtain
\[
\frac{1}{\ell(N)} \sum_{aN < j \le bN} \frac{\ell(j)}{j} f\!\left(\frac{j}{N}\right)
\to \int_{a}^{b} \frac{f(u)}{u}\,\mathrm{d}u.
\]

If $a=0$, suppose that
\[
|f(u)| \le C u^\rho,\quad 0 < u \le u_0,
\]
where $\rho > 0$.
Fix
\[
0 < \delta < \min\{1, u_0, b\}.
\]
Decompose the sum as
\[
\begin{aligned}
S_N(0,b)
&= \frac{1}{\ell(N)} \sum_{1\le j\le \delta N} \frac{\ell(j)}{j} f\!\left(\frac{j}{N}\right)
+ \frac{1}{\ell(N)} \sum_{\delta N< j\le bN} \frac{\ell(j)}{j} f\!\left(\frac{j}{N}\right) \\
&=: I_{N,\delta} + J_{N,\delta}.
\end{aligned}
\]
For fixed $\delta>0$, the result from \eqref{eq:RiemannSV}  implies
\[
J_{N,\delta} \to \int_{\delta}^{b} \frac{f(u)}{u}\,\mathrm{d}u.
\]
Hence it suffices to verify that
\[
\lim_{\delta\downarrow 0}\,\limsup_{N\to\infty} |I_{N,\delta}| = 0.
\]

Since $f(u) = O(u^\rho)$, pick any $0<\varepsilon<\rho$. Fix a sufficiently large integer $j_0$, and decompose the sum as
\[
I_{N,\delta}
= \frac{1}{\ell(N)} \sum_{1\le j<j_0} \frac{\ell(j)}{j} \big|f(j/N)\big|
+ \frac{1}{\ell(N)} \sum_{j_0\le j\le \delta N} \frac{\ell(j)}{j} \big|f(j/N)\big|.
\]
For the finite initial part, using $f(j/N)=O\big((j/N)^\rho\big)$, we have
\[
\frac{1}{\ell(N)} \sum_{1\le j<j_0} \frac{\ell(j)}{j} \big|f(j/N)\big|
= O\left( \frac{N^{-\rho}}{\ell(N)} \right) = o(1).
\]
For the range $j_0\le j\le \delta N$, we apply the Potter bound to obtain rigorously
\[
\frac{\ell(j)}{\ell(N)} \le C_\varepsilon \left(\frac{j}{N}\right)^{-\varepsilon}.
\]
Hence,
\[
\begin{aligned}
\frac{1}{\ell(N)} \sum_{j_0\le j\le \delta N} \frac{\ell(j)}{j} \big|f(j/N)\big|
&\le C N^{\varepsilon-\rho} \sum_{j_0\le j\le \delta N} j^{\rho-\varepsilon-1} \\
&\le C' \delta^{\rho-\varepsilon}.
\end{aligned}
\]
This gives the desired result for the case $a=0$.

Finally, since $L$ diverges, apply
\eqref{eq:deHaan} successively on geometric intervals, or use the
boundary form of Karamata's theorem, to obtain
$\ell(x)=o(L(x))$.  Then
$L(cx)/L(x)\to1$, so $L$ is slowly varying.
\end{proof}

\begin{lemma}[Tempered cumulative sums]
\label{lem:temperedcumulative}
Let
\[
 A_N(x):=\sum_{j=1}^{\floor{Nx}}
 \e^{-\lambda_Nj}\frac{\ell(j)}{j}.
\]
Under \cref{ass:filter,ass:tempering}, for every $x>0$,
\[
 \frac{A_N(x)}{A_N(1)}\longrightarrow1, ~~~~\text{as} ~~N\to \infty.
\]
If $\lambda_*<\infty$, then $A_N(1)\sim L(N)$.  If
$\lambda_*=\infty$, then
$A_N(1)\sim L(1/\lambda_N)$.
\end{lemma}

\begin{proof}
Suppose first that $N\lambda_N$ is bounded.  For fixed $\eta\in(0,1)$,
\[
 0\le L(N)-A_N(1)
 \le \bigl(1-\e^{-C\eta}\bigr)L(\eta N)
     +\{L(N)-L(\eta N)\},
\]
for all sufficiently large $N$ and a fixed $C<\infty$.  Divide by
$L(N)$, use slow variation, and then let $\eta\downarrow0$.  This
proves $A_N(1)\sim L(N)$.  Moreover, for fixed $x>0$,
\[
 |A_N(x)-A_N(1)|=O(\ell(N))
\]
by Lemma \ref{lem:integratedSV}, and the ratio conclusion follows from
\eqref{eq:elloverL}.

Now suppose $N\lambda_N\to\infty$.  Regarding $L$ as a
right-continuous nondecreasing function, the Abelian theorem for
Laplace--Stieltjes transforms yields
\[
 \sum_{j=1}^{\infty}\e^{-\lambda_Nj}\frac{\ell(j)}{j}
 =\int_{[1,\infty)}\e^{-\lambda_Nx}\dd L(x)
 \sim L(1/\lambda_N).
\]
For every fixed $x>0$, $Nx\lambda_N\to\infty$.  Potter bounds and an
exponential-tail estimate show that the part with $j>Nx$ is
$o(L(1/\lambda_N))$.  Hence both $A_N(x)$ and $A_N(1)$ are
asymptotic to $L(1/\lambda_N)$.
\end{proof}

\begin{lemma}[Moderate second-order expansion]
\label{lem:secondexpansion}
Assume $N\lambda_N\to\lambda\in[0,\infty)$.  For every $x>0$,
\[
 \frac{A_N(x)-L(N)}{\ell(N)}
 \overset{N\to\infty}{\longrightarrow} G_\lambda(x),
\]
where
\[
 G_\lambda(x)
 =\log x+\int_0^x\frac{\e^{-\lambda u}-1}{u}\dd u.
\]
The convergence is locally uniform for $x$ in compact subsets of
$(0,\infty)$.
\end{lemma}

\begin{proof}
Write
\[
 A_N(x)-L(N)
 =\{L(Nx)-L(N)\}
  +\sum_{j\le Nx}
    \bigl(\e^{-\lambda_Nj}-1\bigr)\frac{\ell(j)}{j}.
\]
The first term divided by $\ell(N)$ tends to $\log x$ by
\eqref{eq:deHaan}.  Apply \eqref{eq:RiemannSV} to the second term
away from zero.  Near zero,
$|\e^{-\lambda_Nj}-1|\le Cj/N$, and Potter's bound supplies an
integrable majorant.  Thus the second term converges to
$\int_0^x(\e^{-\lambda u}-1)u^{-1}\dd u$.  The same estimates are
uniform when $x$ stays in a compact subset of $(0,\infty)$.
\end{proof}

\subsection{Cluster compression and remote-past estimates}

For $m\ge0$ and $T>0$, put
\[
 D_N(m,T):=
 \sum_{r=m+1}^{m+\floor{NT}}b_{N,r}.
\]
This is the coefficient with which the innovation at time $-m$
enters the partial sums observed over $[0,T]$.

\begin{lemma}[Remote past]
\label{lem:remotepast}
For every $p>1$ and $T<\infty$, as $N\to\infty$,
\[
 \frac1N\sum_{m=0}^{\infty}
 \left(\frac{D_N(m,T)}{Q_N}\right)^p
 \longrightarrow0.
\]
\end{lemma}

\begin{proof}
We give the estimates separately because they explain the phase
transition.  If $N\lambda_N$ is bounded, then $Q_N\sim L(N)$.
For $m\le\delta N$, the summand is uniformly bounded, and this part
is at most $C\delta$.  For $m>\delta N$, positivity and Potter's
bound give, with any sufficiently small $\eta>0$,
\[
 D_N(m,T)
 \le C N\frac{\ell(m)}{m},
 \qquad
 \frac{\ell(m)}{\ell(N)}
 \le C\max\{(m/N)^\eta,(m/N)^{-\eta}\}.
\]
Choose $\eta<1-1/p$.  Comparing the remaining sum with an integral,
\[
 \frac1N\sum_{m>\delta N}
 \left(\frac{D_N(m,T)}{Q_N}\right)^p
 \le
 C_\delta
 \left(\frac{\ell(N)}{L(N)}\right)^p
 \int_\delta^\infty u^{-p+p\eta}\dd u
 \longrightarrow0.
\]
Letting $\delta\downarrow0$ proves the assertion.

If $N\lambda_N\to\infty$, write $r_N=\lambda_N^{-1}=o(N)$.
Now $Q_N\sim L(r_N)$.  The range $m\le Kr_N$ contributes at most
$CKr_N/N=o(1)$.  For $m>Kr_N$, the exponential factor and Potter's
bound yield
\[
 D_N(m,T)
 \le C Q_N\,\e^{-c m/r_N}(m/r_N)^\eta.
\]
Consequently, the remaining part is bounded by
\[
 C\frac{r_N}{N}
 \int_K^\infty \e^{-cpu}u^{p\eta}\dd u=o(1).
\]
\end{proof}

\begin{lemma}[Compressed cluster]
\label{lem:compressedcluster}
For every fixed $0<T<\infty$, there exists a sequence of positive
integers $m_N$ such that
\[
m_N\longrightarrow\infty,
\qquad
\frac{m_N}{N}\longrightarrow0,
\]
and
\begin{equation}
 \frac1{Q_N}\sum_{j=1}^{m_N}b_{N,j}
 \xrightarrow[N\to\infty]{}1,
 \qquad
 \frac1{Q_N}\sum_{m_N<j\le \floor{NT}}b_{N,j}
 \xrightarrow[N\to\infty]{}0.
 \label{eq:clustercompression}
\end{equation}
\end{lemma}

\begin{proof}
Recall that
\[
 A_N(x)=\sum_{j=1}^{\floor{Nx}}b_{N,j},
 \qquad
 Q_N=A_N(1),
\]
and that $b_{N,j}\ge0$. We treat the bounded and strong tempering
regimes separately.

\medskip
\noindent
\textit{Case 1: $N\lambda_N$ is bounded.}
By Lemma~\ref{lem:temperedcumulative}, for every fixed $c>0$,
\begin{equation}
 \frac{A_N(c)}{Q_N}\longrightarrow1.
 \label{eq:fixed-c-ratio}
\end{equation}
We now give the diagonal construction explicitly. Let
\[
 d_k:=\min\left\{\frac{T}{2},\frac1k\right\},
 \qquad k\ge1.
\]
Then $d_k\downarrow0$. For each $k$, relation
\eqref{eq:fixed-c-ratio}, applied with $c=d_k$, allows us to choose
strictly increasing integers $N_k$ such that
\begin{equation}
 \left|\frac{A_N(d_k)}{Q_N}-1\right|\le\frac1k
 \quad\text{for every }N\ge N_k,
 \qquad
 N_kd_k\ge 2k.
 \label{eq:diagonal-choice}
\end{equation}
For $N_k\le N<N_{k+1}$, set
\[
 c_N:=d_k,
 \qquad
 m_N:=\floor{Nc_N}.
\]
If $k=k(N)$ denotes the corresponding block index, then
$k(N)\to\infty$ and hence
\[
 c_N=d_{k(N)}\longrightarrow0.
\]
Moreover, by \eqref{eq:diagonal-choice},
\[
 m_N
 \ge Nc_N-1
 \ge N_{k(N)}d_{k(N)}-1
 \ge 2k(N)-1
 \longrightarrow\infty,
\]
whereas
\[
 0\le\frac{m_N}{N}\le c_N\longrightarrow0.
\]
The first relation in \eqref{eq:clustercompression} now follows from
\[
 \frac1{Q_N}\sum_{j=1}^{m_N}b_{N,j}
 =\frac{A_N(c_N)}{Q_N},
\]
because \eqref{eq:diagonal-choice} gives
\[
 \left|\frac{A_N(c_N)}{Q_N}-1\right|
 \le\frac1{k(N)}
 \longrightarrow0.
\]

Since $c_N\le T/2$, we have $m_N<\floor{NT}$ for all sufficiently
large $N$. Therefore,
\[
 \frac1{Q_N}\sum_{m_N<j\le\floor{NT}}b_{N,j}
 =
 \frac{A_N(T)-A_N(c_N)}{Q_N}.
\]
Lemma~\ref{lem:temperedcumulative} gives
$A_N(T)/Q_N\to1$, while the preceding argument gives
$A_N(c_N)/Q_N\to1$. Consequently,
\[
 \frac{A_N(T)-A_N(c_N)}{Q_N}\longrightarrow1-1=0,
\]
which proves the second relation.

\medskip
\noindent
\textit{Case 2: $N\lambda_N\to\infty$.}
Put
\[
 h_N:=N\lambda_N,
 \qquad
 \underline h_N:=\inf_{k\ge N}h_k.
\]
Because $h_N\to\infty$, the sequence $\underline h_N$ is
nondecreasing and $\underline h_N\to\infty$. Define, for all
sufficiently large $N$,
\[
 c_N:=
 \min\left\{\frac{T}{2},\underline h_N^{-1/2}\right\},
 \qquad
 m_N:=\floor{Nc_N}.
\]
Then $c_N\downarrow0$. Eventually
$c_N=\underline h_N^{-1/2}$, and hence
\[
 c_NN\lambda_N
 =\frac{h_N}{\sqrt{\underline h_N}}
 \ge\sqrt{\underline h_N}
 \longrightarrow\infty.
 \label{eq:beyond-tempering-scale}
\]
Furthermore, since $\underline h_N\le h_N=N\lambda_N$,
\[
 Nc_N
 =\frac{N}{\sqrt{\underline h_N}}
 \ge\frac{N}{\sqrt{N\lambda_N}}
 =\sqrt{\frac{N}{\lambda_N}}
 \longrightarrow\infty.
\]
It follows that
\[
 m_N\longrightarrow\infty,
 \qquad
 \frac{m_N}{N}\le c_N\longrightarrow0.
\]

It remains to justify carefully that truncation at $m_N$ captures
asymptotically all of the tempered coefficient mass. Set
\[
 r_N:=\lambda_N^{-1}.
\]
Then $r_N\to\infty$, $r_N=o(N)$, and by
$c_N N \lambda_N\to\infty$ and $ m_N:=\floor{Nc_N}$,
\[
 u_N:=\frac{m_N}{r_N}=m_N\lambda_N
 \longrightarrow\infty.
 \label{eq:uN-infinity}
\]
Thus $m_N$ lies far beyond the effective tempering scale
$r_N=1/\lambda_N$.

The Abelian estimate used in the proof of
Lemma~\ref{lem:temperedcumulative} yields
\begin{equation}
 H_N:=\sum_{j=1}^{\infty}b_{N,j}
 \sim L(r_N),
 \qquad
 Q_N\sim L(r_N).
 \label{eq:full-tempered-mass}
\end{equation}
Consider the tail beyond $m_N$,
\[
 \mathcal{R}_N:=\sum_{j>m_N}b_{N,j}
 =\sum_{j>m_N}e^{-j/r_N}\frac{\ell(j)}{j}.
\]
Fix $\varepsilon\in(0,1)$. Since $u_N\to\infty$, we have
$j/r_N\ge1$ whenever $j>m_N$ and $N$ is sufficiently large.
Potter's bound therefore gives
\[
 \frac{\ell(j)}{\ell(r_N)}
 \le C_\varepsilon
 \left(\frac{j}{r_N}\right)^\varepsilon,
 \qquad j>m_N.
\]
Consequently,
\begin{align*}
\mathcal{R}_N
 &\le
 C_\varepsilon\ell(r_N)
 \sum_{j>m_N}
 e^{-j/r_N}\frac1j
 \left(\frac{j}{r_N}\right)^\varepsilon                                      
 \frac{C_\varepsilon\ell(r_N)}{r_N}
 \sum_{j>m_N}
 e^{-j/r_N}
 \left(\frac{j}{r_N}\right)^{\varepsilon-1}                                    \\
 &\le
 C_\varepsilon\ell(r_N)
 \int_{u_N/2}^{\infty}e^{-v}v^{\varepsilon-1}\,dv.
 \label{eq:strong-tail-bound}
\end{align*}
The last integral tends to zero because $u_N\to\infty$. Hence
\[
 \mathcal{R}_N=o\bigl(\ell(r_N)\bigr).
\]
By Lemma~\ref{lem:integratedSV},
$\ell(r_N)=o(L(r_N))$, so in particular
\[
 \mathcal{R}_N=o\bigl(L(r_N)\bigr).
 \label{eq:tail-negligible}
\]
Combining \eqref{eq:full-tempered-mass} with the estimate
$
 \mathcal{R}_N=o(L(r_N)),
$
we obtain
\[
 A_N(c_N)
 =\sum_{j=1}^{m_N}b_{N,j}
 =H_N-\mathcal{R}_N
 \sim L(r_N)
 \sim Q_N.
\]
This proves
\[
 \frac1{Q_N}\sum_{j=1}^{m_N}b_{N,j}
 =\frac{A_N(c_N)}{Q_N}
 \longrightarrow1.
\]

Finally, $c_N\le T/2$ implies $m_N<\floor{NT}$ eventually. Hence
\[
 \frac1{Q_N}\sum_{m_N<j\le\floor{NT}}b_{N,j}
 =
 \frac{A_N(T)-A_N(c_N)}{Q_N}.
\]
Lemma~\ref{lem:temperedcumulative} gives
$A_N(T)/Q_N\to1$, and we have just proved
$A_N(c_N)/Q_N\to1$. Their difference therefore converges to zero.
This proves the second relation and completes the proof.
\end{proof}

\begin{lemma}[Block coefficients]
\label{lem:blockcoeff}
Let $0=t_0<t_1<\cdots<t_m\le T$.  For $1\le k\le m$ and
$j\in\Z$, set
\[
 C_{N,k}(j):=
 \sum_{n=(j+1)\vee(\floor{Nt_{k-1}}+1)}^{\floor{Nt_k}}
 b_{N,n-j},
\]
with an empty sum equal to zero.  For every $\delta>0$, when $N\to\infty$,
\begin{align}
 &\sup_{\floor{N(t_{k-1}+\delta)}\le j
          \le\floor{N(t_k-\delta)}}
 \left|\frac{C_{N,k}(j)}{Q_N}-1\right|
 \longrightarrow0,
 \label{eq:mainblock}\\
 &\sup_{\floor{N(t_{k-1}+\delta)}\le j
          \le\floor{N(t_k-\delta)}}
 \frac{C_{N,i}(j)}{Q_N}
 \longrightarrow0,\qquad i>k.
 \label{eq:crossblock}
\end{align}
Furthermore, the indices within distance $\delta N$ of the block
endpoints contribute at most $C\delta$ to the normalized empirical
$p$th powers, uniformly in $N$, for every $p>1$.
\end{lemma}

\begin{proof}
For $j$ in the interior of the $k$th block,
\[
 C_{N,k}(j)=A_N\bigl(t_k-j/N+o(1)\bigr),
\]
whose ratio to $Q_N$ tends uniformly to one by Lemma \ref{lem:temperedcumulative}.  If $i>k$, then
\[
 C_{N,i}(j)
 =A_N\bigl(t_i-j/N+o(1)\bigr)
  -A_N\bigl(t_{i-1}-j/N+o(1)\bigr),
\]
and both cumulative sums are asymptotic to $Q_N$, uniformly because
their arguments stay in a compact subset of $(0,\infty)$.
The endpoint estimate follows from positivity, the uniform bound
$A_N(T)/Q_N\le C_T$, and the fact that the endpoint strips contain
at most $C\delta N$ indices.
\end{proof}

\subsection{Weighted stable sums and the logarithmic kernel}

The following standard triangular-array result is stated in the form
needed below.  It follows from characteristic functions, or from the
random-measure convergence associated with \eqref{eq:iid-fclt}; see
\citet{KasaharaMaejima1988}.

\begin{lemma}[Weighted stable array]
\label{lem:weightedarray}
Let $w_{N,j}\in\mathbb R$, $j\in\mathbb Z$, and define
$
 w_N(y):=w_{N,\floor{Ny}},  y\in\mathbb R.
$
Suppose that, for some $w\in L^\alpha(\mathbb R)$,
\begin{equation}
 \int_{\mathbb R}|w_N(y)-w(y)|^\alpha\,dy\longrightarrow0,
 \label{eq:Lalphaconv}
\end{equation}
and
\begin{equation}
 \frac{\max_{j\in\mathbb Z}|w_{N,j}|}{B_N}\longrightarrow0.
 \label{eq:maxweight}
\end{equation}
Assume in addition that, for some
$0<\eta<\alpha-1$,
\begin{equation}
 \sup_{N\ge1}\int_{\mathbb R}
 \left\{|w_N(y)|^{\alpha-\eta}
       +|w_N(y)|^{\alpha+\eta}\right\}\,dy<\infty.
 \label{eq:neighborLp}
\end{equation}
Then, under Assumption~\ref{ass:innovations},
\[
 \frac1{B_N}\sum_{j\in\mathbb Z}w_{N,j}\varepsilon_j
 \xrightarrow[N\to\infty]{d}
 \int_{\mathbb R}w(y)M_\alpha(dy).
\]
For $\alpha=2$, the same result holds with $L^2$ in
\eqref{eq:Lalphaconv}, with $0<\eta<1$ in
\eqref{eq:neighborLp}, and with $M_2$ denoting the Gaussian random
measure.
\end{lemma}
\begin{proof}
Let
\[
    \varphi(v):=\E e^{iv\varepsilon_0},
    \qquad
    q(v):=\varphi(v)-1.
\]
Since $\E\varepsilon_0=0$,
\[
 q(v)
 =
 \E\bigl(e^{iv\varepsilon_0}-1-iv\varepsilon_0\bigr).
\]
We first work with $q(v)$ and only afterwards return to
$\log\varphi(v)$.  This avoids a nonuniform remainder in the
expansion of the logarithm.

Suppose first that $1<\alpha<2$.  The inequality
\[
 |e^{ix}-1-ix|
 \leq C\min\{|x|^2,|x|\}
\]
together with Karamata's theorem gives
\[
 |q(v)|
 \leq C|v|^\alpha h(1/|v|),
 \qquad |v|\leq v_0.
\]
Indeed, splitting at $|\varepsilon_0|=|v|^{-1}$, the truncated part
is bounded by
\[
 Cv^2\E\left[
       \varepsilon_0^2
       \mathbf 1_{\{|\varepsilon_0|\leq |v|^{-1}\}}
     \right]
 \leq C|v|^\alpha h(1/|v|),
\]
whereas the remaining part is bounded by
\[
 C|v|\E\left[
       |\varepsilon_0|
       \mathbf 1_{\{|\varepsilon_0|>|v|^{-1}\}}
     \right]
 +C\Pp\bigl(|\varepsilon_0|>|v|^{-1}\bigr)
 \leq C|v|^\alpha h(1/|v|).
\]

For $\alpha=2$, the corresponding standard normal-domain-of-attraction
estimate is
\[
 |q(v)|
 \leq Cv^2V(1/|v|),
 \qquad
 V(x):=\E\left[
          \varepsilon_0^2
          \mathbf 1_{\{|\varepsilon_0|\leq x\}}
        \right].
\]
Choose $0<\eta_0<\eta$,  from the
normalization of $B_N$ and Potter's bound, it follows that
\[
 N\left|q\left(\frac{u}{B_N}\right)\right|
 \leq C
 \left(
   |u|^{\alpha-\eta_0}
   +|u|^{\alpha+\eta_0}
 \right)
\]
uniformly whenever $|u|/B_N=o(1)$.  For $\alpha=2$, the same argument
gives the preceding estimate with $\alpha$ replaced by $2$.

Define
\[
 \widetilde F_N(u)
 :=
 Nq(u/B_N).
\]
The domain-of-attraction assumption gives
\[
 N\log\varphi(u/B_N)\longrightarrow\Psi_\alpha(u)
\]
locally uniformly in $u$, where $\Psi_\alpha$ is the characteristic
exponent of $Z_\alpha(1)$.  Since, uniformly on compact sets,
$q(u/B_N)=O(N^{-1})$ and
\[
 \log(1+z)=z+O(|z|^2),
\]
we also have
\[
 \widetilde F_N(u)\longrightarrow\Psi_\alpha(u)
\]
locally uniformly in $u$.  Moreover,
\[
 |\Psi_\alpha(u)|\leq C|u|^\alpha.
\]

Fix $\theta\in\mathbb R$ and put
\[
 z_{N,j}:=\frac{\theta w_{N,j}}{B_N}.
\]
Condition \eqref{eq:maxweight} implies
\[
 \max_{j\in\mathbb Z}|z_{N,j}|\longrightarrow0.
\]
Consequently,
\begin{align*}
 \sum_{j\in\mathbb Z}|q(z_{N,j})|
 &=
 \int_{\mathbb R}
 N\left|
   q\left(\frac{\theta w_N(y)}{B_N}\right)
 \right|dy                                                     \\
 &\leq
 C_\theta\int_{\mathbb R}
 \left(
   |w_N(y)|^{\alpha-\eta_0}
   +|w_N(y)|^{\alpha+\eta_0}
 \right)dy                                                     \\
 &\leq C_\theta.
\end{align*}
The last bound follows from  \eqref{eq:neighborLp} on
$\{|w_N|\leq1\}$ the integrand is controlled by
$|w_N|^{\alpha-\eta}$, while on $\{|w_N|>1\}$ it is controlled by
$|w_N|^{\alpha+\eta}$.

We now justify the replacement of $\log\varphi$ by $q$.  Since
$\max_j|z_{N,j}|\to0$,
\[
 \max_j|q(z_{N,j})|\longrightarrow0.
\]
For sufficiently large $N$, $|q(z_{N,j})|\leq1/2$ for all $j$, and
therefore
\begin{align*}
 \left|
   \sum_j\log\varphi(z_{N,j})
   -
   \sum_jq(z_{N,j})
  \right|                                                     
\leq
 C\sum_j|q(z_{N,j})|^2                                       
\leq
 C\max_j|q(z_{N,j})|
   \sum_j|q(z_{N,j})|
 \longrightarrow0.
\end{align*}
Thus it remains to prove
\[
 \int_{\mathbb R}
   \widetilde F_N(\theta w_N(y))\,dy
 \longrightarrow
 \int_{\mathbb R}
   \Psi_\alpha(\theta w(y))\,dy.
\]

For $M>1$, define
\[
 D_{N,M}
 :=
 \left\{
  y\in\mathbb R:
  M^{-1}\leq|w_N(y)|\leq M
 \right\}.
\]
By  \eqref{eq:neighborLp},
\[
 |D_{N,M}|
 \leq
 M^{\alpha-\eta}
 \int_{\mathbb R}|w_N(y)|^{\alpha-\eta}\,dy
 \leq C_M.
\]
The local uniform convergence of $\widetilde F_N$ therefore gives
\[
 \int_{D_{N,M}}
 \left|
   \widetilde F_N(\theta w_N(y))
   -
   \Psi_\alpha(\theta w_N(y))
 \right|dy
 \longrightarrow0.
\]

It remains to control the regions where $w_N$ is very large or very
small.  Taking, for example, $\eta_0=\eta/2$, we obtain
\[
 \sup_N
 \int_{\{|w_N|>M\}}
 \left|
  \widetilde F_N(\theta w_N(y))
 \right|dy
 \leq
 C_\theta M^{-\eta/2},
\]
and
\[
 \sup_N
 \int_{\{|w_N|<M^{-1}\}}
 \left|
  \widetilde F_N(\theta w_N(y))
 \right|dy
 \leq
 C_\theta M^{-\eta/2}.
\]
Indeed, on $\{|w_N|>M\}$,
\[
 |w_N|^{\alpha+\eta/2}
 \leq
 M^{-\eta/2}|w_N|^{\alpha+\eta},
\]
while on $\{|w_N|<M^{-1}\}$,
\[
 |w_N|^{\alpha-\eta/2}
 \leq
 M^{-\eta/2}|w_N|^{\alpha-\eta}.
\]
The other power is smaller on the corresponding region.  Since
$|\Psi_\alpha(u)|\leq C|u|^\alpha$, the same two estimates hold with
$\widetilde F_N$ replaced by $\Psi_\alpha$.

Letting first $N\to\infty$ and then $M\to\infty$, we conclude that
\[
 \int_{\mathbb R}
 \left|
  \widetilde F_N(\theta w_N(y))
  -
  \Psi_\alpha(\theta w_N(y))
 \right|dy
 \longrightarrow0.
\]
Finally, $w_N\to w$ in $L^\alpha(\mathbb R)$ implies
\[
 \int_{\mathbb R}
   \Psi_\alpha(\theta w_N(y))\,dy
 \longrightarrow
 \int_{\mathbb R}
   \Psi_\alpha(\theta w(y))\,dy.
\]
For example, this follows from
\[
 |\Psi_\alpha(a)-\Psi_\alpha(b)|
 \leq
 C|a-b|
 \bigl(|a|^{\alpha-1}+|b|^{\alpha-1}\bigr)
\]
and Hölder's inequality.

Combining the preceding estimates gives
\[
 \log\E\exp\left\{
  \frac{i\theta}{B_N}
  \sum_{j\in\mathbb Z}w_{N,j}\varepsilon_j
 \right\}
 \longrightarrow
 \int_{\mathbb R}
   \Psi_\alpha(\theta w(y))\,dy.
\]
The expression on the right is the logarithm of the characteristic
function of
\[
 \int_{\mathbb R}w(y)\,M_\alpha(dy).
\]
L\'{e}vy's continuity theorem proves the desired convergence.

For completeness, the infinite weighted sums are well defined.  Put
$q_0=\alpha-\eta>1$.  Then
$\E|\varepsilon_0|^{q_0}<\infty$ and
\[
 \sum_{j\in\mathbb Z}|w_{N,j}|^{q_0}
 =
 N\int_{\mathbb R}|w_N(y)|^{q_0}\,dy
 <\infty.
\]
The von Bahr--Esseen inequality shows that the truncated weighted sums
form a Cauchy sequence in $L^{q_0}$.  The proof for $\alpha=2$ is
identical, using the truncated variance $V$ and the Gaussian exponent
$\Psi_2(u)=-u^2/2$.
\end{proof}

For $t\ge0$ and $i\in\Z$, define the centered second-order coefficient
\begin{equation}
 k_{N,t}(i):=
 \frac1{\ell(N)}
 \left[
 \sum_{n=1}^{\floor{Nt}}
 b_{N,n-i}\1_{\{n-i\ge1\}}
 -L(N)\1_{\{1\le i\le\floor{Nt}\}}
 \right].
 \label{eq:discretekernel}
\end{equation}

\begin{lemma}[Kernel convergence]
\label{lem:kernelconv}
Assume $N\lambda_N\to\lambda<\infty$.  For every
$t\in[0,T]$ and every $r>1$,
\[
 \int_\R
 \left|k_{N,t}(\floor{Ny})-K_\lambda(t,y)\right|^r
 \dd y\longrightarrow0,
\]
as $N$ tends to infinity. The same holds for every finite linear combination in $t$.
Moreover,
\begin{equation}
 \max_{i\in\Z}|k_{N,t}(i)|=N^{o(1)}.
 \label{eq:kernelmax}
\end{equation}
\end{lemma}
\begin{proof}
Write
\[
 n_t:=\floor{Nt},
 \qquad
 H_N(x):=\frac{A_N(x)-L(N)}{\ell(N)},\qquad x\ge0,
\]
where $A_N(0)=0$. From the definition of $k_{N,t}$, we have the
exact representation
\begin{equation}
 k_{N,t}(i)=
 \begin{cases}
 H_N\bigl((n_t-i)/N\bigr),
     &1\le i\le n_t,\\[1mm]
 H_N\bigl((n_t-i)/N\bigr)-H_N(-i/N),
     &i\le0,\\[1mm]
 0,&i>n_t.
 \end{cases}
 \label{eq:exact-kernel-representation}
\end{equation}

We first prove pointwise convergence. Let $0<y<t$ and put
$i_N=\floor{Ny}$. Then, for all sufficiently large $N$,
$1\le i_N\le n_t$, and
\[
 \frac{n_t-i_N}{N}\longrightarrow t-y.
\]
The locally uniform convergence in
Lemma~\ref{lem:secondexpansion} therefore gives
\[
 k_{N,t}(i_N)
 =
 H_N\left(\frac{n_t-i_N}{N}\right)
 \longrightarrow G_\lambda(t-y).
\]
If $y<0$, then $i_N\le0$ eventually, and
\[
 \frac{n_t-i_N}{N}\longrightarrow t-y,
 \qquad
 \frac{-i_N}{N}\longrightarrow-y.
\]
Applying Lemma~\ref{lem:secondexpansion} to both terms in
\eqref{eq:exact-kernel-representation}, we obtain
\[
 k_{N,t}(i_N)
 \longrightarrow
 G_\lambda(t-y)-G_\lambda(-y).
\]
If $y>t$, then $i_N>n_t$ eventually and $k_{N,t}(i_N)=0$.
Consequently,
\begin{equation}
 k_{N,t}(\floor{Ny})
 \longrightarrow K_\lambda(t,y)
 \label{eq:kernel-pointwise}
\end{equation}
for every $y\notin\{0,t\}$. The case $t=0$ is immediate because
both sides vanish identically.

We next establish uniform integrability. Since
$N\lambda_N\to\lambda<\infty$, there exists $C<\infty$ such that
$N\lambda_N\le C$. For $N^{-1}\le x\le T+1$, write
\[
 A_N(x)-L(N)
 =
 L(Nx)-L(N)
 +
 \sum_{j\le Nx}
 \bigl(e^{-\lambda_Nj}-1\bigr)\frac{\ell(j)}j.
\]
Let $\eta>0$. Potter's bound gives
\[
 \frac{|L(Nx)-L(N)|}{\ell(N)}
 \le C_{\eta,T}\bigl(1+x^{-\eta}\bigr).
\]
Furthermore,
\[
 |e^{-\lambda_Nj}-1|
 \le \lambda_Nj
 \le C\frac jN,
\]
and hence
\[
 \frac1{\ell(N)}
 \sum_{j\le Nx}
 |e^{-\lambda_Nj}-1|\frac{\ell(j)}j
 \le
 \frac C{N\ell(N)}\sum_{j\le Nx}\ell(j)
 \le C_{\eta,T}.
\]
Thus
\begin{equation}
 |H_N(x)|
 \le C_{\eta,T}\bigl(1+x^{-\eta}\bigr),
 \qquad N^{-1}\le x\le T+1.
 \label{eq:HN-near-zero}
\end{equation}

For the remote past, let $x\ge1$ and $0\le v\le T+1$. Positivity,
the bound $e^{-\lambda_Nj}\le1$, and Potter's inequality yield
\begin{align}
 \frac{|A_N(x+v)-A_N(x)|}{\ell(N)}
 &\le
 \frac1{\ell(N)}
 \sum_{Nx<j\le N(x+v)+1}\frac{\ell(j)}j \notag\\
 &\le
 C_{\eta,T}\int_x^{x+v+N^{-1}}u^{-1+\eta}\,du \notag\\
 &\le C_{\eta,T}(1+x)^{-1+\eta}.
 \label{eq:remote-kernel-bound}
\end{align}

The cells containing $0$ or $t$ require separate treatment because
the corresponding arguments of $H_N$ can be zero. Let $E_N(t)$ be
the union of these cells. Then
\[
 |E_N(t)|\le\frac4N.
\]
We claim that
\begin{equation}
 \sup_{i\in\mathbb Z}|k_{N,t}(i)|
 \le C_T\left(1+\frac{L(N)}{\ell(N)}\right).
 \label{eq:preliminary-max-bound}
\end{equation}
Indeed, if $1\le i\le n_t$, then
\[
 |k_{N,t}(i)|
 \le\frac{A_N(T)+L(N)}{\ell(N)}
 \le C_T\frac{L(N)}{\ell(N)}.
\]
For $i\le0$, put $m=-i$. If $m\le N$, the corresponding coefficient
is bounded by $A_N(T+1)/\ell(N)$, which is at most
$C_TL(N)/\ell(N)$. If $m>N$, Potter's inequality gives
\begin{align*}
 |k_{N,t}(i)|
 \le
 \frac1{\ell(N)}
 \sum_{m<j\le m+n_t}\frac{\ell(j)}j
 \le
 C_T\frac{N\ell(m)}{m\ell(N)}
 \le
 C_T\left(\frac mN\right)^{-1+\eta}
 \le C_T
\end{align*}
after choosing $\eta<1$. This proves
\eqref{eq:preliminary-max-bound}.

By Lemma~\ref{lem:integratedSV}, both $L$ and $\ell$ are slowly
varying, and hence
\[
 1+\frac{L(N)}{\ell(N)}=N^{o(1)}.
\]
It follows from \eqref{eq:preliminary-max-bound} that
\begin{equation}
 \int_{E_N(t)}
 |k_{N,t}(\floor{Ny})|^r\,dy
 \le
 \frac{C_T}{N}
 \left(1+\frac{L(N)}{\ell(N)}\right)^r
 \longrightarrow0.
 \label{eq:endpoint-cells}
\end{equation}

Now choose
\[
 0<\eta<
 \min\left\{\frac1r,\,1-\frac1r\right\}.
\]
Relations \eqref{eq:HN-near-zero} and
\eqref{eq:remote-kernel-bound} imply, outside $E_N(t)$,
\[
 |k_{N,t}(\floor{Ny})|
 \le C_{\eta,T}\,g_{\eta,t}(y),
\]
where one may take
\[
 g_{\eta,t}(y)
 :=
 \mathbf 1_{\{0<y<t\}}
       \bigl[1+(t-y)^{-\eta}\bigr]
 +\mathbf 1_{\{-1<y<0\}}
       \bigl[1+(-y)^{-\eta}\bigr]
 +\mathbf 1_{\{y\le-1\}}
       (1+|y|)^{-1+\eta}.
\]
Our choice of $\eta$ ensures that
\[
 g_{\eta,t}\in L^r(\mathbb R).
\]
The limiting kernel satisfies the same bound. Indeed,
\[
 G_\lambda'(x)=\frac{e^{-\lambda x}}x,
\]
so its singularities are logarithmic and its remote-past difference
is of order $|y|^{-1}$ when $\lambda=0$, and is even smaller when
$\lambda>0$.

Combining the pointwise convergence
\eqref{eq:kernel-pointwise}, dominated convergence outside
$E_N(t)$, and \eqref{eq:endpoint-cells}, we obtain
\[
 \int_{\mathbb R}
 \left|
 k_{N,t}(\floor{Ny})-K_\lambda(t,y)
 \right|^r\,dy
 \xrightarrow[N\to\infty]{}0.
\]

For a finite linear combination
\[
 \sum_{\nu=1}^m\theta_\nu
 k_{N,t_\nu}(\floor{Ny}),
\]
the exceptional set is the union of the cells containing
$0,t_1,\ldots,t_m$, and therefore still has measure $O(N^{-1})$.
The preceding majorants can be summed over the finitely many
$t_\nu$, so the same argument proves the asserted convergence for
every finite linear combination.

Finally, \eqref{eq:preliminary-max-bound} and slow variation give
\[
 \max_{i\in\mathbb Z}|k_{N,t}(i)|
 \le
 C_T\left(1+\frac{L(N)}{\ell(N)}\right)
 =N^{o(1)}.
\]
This proves \eqref{eq:kernelmax} and completes the proof.
\end{proof}

Since the second-order limit theorem only considers the effects of weak or moderate tempering, the following technical lemma used in the proof of Theorem \ref{thm:secondorder} is also restricted to the premise of condition \eqref{eq:secondregime}.

\begin{lemma}[Gaussian increment bound]
\label{lem:gaussiantight}
Suppose that $\alpha=2$ and
$\mathbb E|\varepsilon_0|^{2+\delta}<\infty$ for some $\delta>0$.
Let $\widetilde R_N$ denote the polygonal interpolation of $R_N$.
Choose
\[
 2<p\le2+\delta
\]
and then choose
\[
 0<\vartheta<\frac12-\frac1p.
\]
There exists a constant $C_{T,p,\vartheta}$ such that, uniformly in
$N$ and $s,t\in[0,T]$,
\begin{equation}
 \mathbb E
 |\widetilde R_N(t)-\widetilde R_N(s)|^p
 \le
 C_{T,p,\vartheta}
 \bigl(|t-s|+N^{-1}\bigr)^{p(1/2-\vartheta)}.
 \label{eq:Gaussianincrement-corrected}
\end{equation}
Moreover, the maximal polygonal interpolation error converges to
zero in probability.
\end{lemma}

\begin{proof}
Since $\mathbb E|\varepsilon_0|^{2+\delta}<\infty$, we have
\[
 \sigma_\varepsilon^2:=\mathbb E\varepsilon_0^2<\infty.
\]
Consequently, the Gaussian domain-of-attraction normalization
satisfies
\begin{equation}
 B_N^2\sim N\sigma_\varepsilon^2.
 \label{eq:BN-finite-variance}
\end{equation}

We first consider two grid points. Let $0\le a<b\le\floor{NT}$,
put
\[
 q:=b-a,\qquad h:=\frac qN,
\]
and define
\[
 D_{N,a,b}
 :=
 R_N(b/N)-R_N(a/N).
\]
Using the exact linear representation of $R_N$ and shifting the
innovation index by $a$, we obtain
\begin{equation}
 D_{N,a,b}
 =
 \frac1{B_N}\sum_{j\in\mathbb Z}d_{N,q}(j)\varepsilon_{a+j},
 \label{eq:grid-increment-representation}
\end{equation}
where
\[
 d_{N,q}(j)
 :=
 \frac1{\ell(N)}
 \left[
 \sum_{n=1}^{q}
 b_{N,n-j}\mathbf 1_{\{n-j\ge1\}}
 -
 L(N)\mathbf 1_{\{1\le j\le q\}}
 \right].
\]
Thus $d_{N,q}(j)=k_{N,q/N}(j)$.

We next derive a deterministic square-sum estimate. Fix
$0<\eta_0<1/2$. Potter's inequality and the estimates used in the
proof of Lemma~\ref{lem:kernelconv} imply
\begin{equation}
 \frac1N\sum_{j\in\mathbb Z}|d_{N,q}(j)|^2
 \le
 C_{\eta_0,T}
 \left(h+\frac1N\right)^{1-2\eta_0}.
 \label{eq:coefficient-square-bound}
\end{equation}
For completeness, we verify the different ranges.

If $1\le j\le q$, put $u_j=(q-j)/N$. Then
\[
 |d_{N,q}(j)|
 \le
 C_{\eta_0,T}
 \left(u_j+\frac1N\right)^{-\eta_0}.
\]
The possible value $L(N)/\ell(N)$ at $u_j=0$ is absorbed because
$L(N)/\ell(N)=N^{o(1)}$. Therefore,
\begin{align*}
 \frac1N\sum_{j=1}^{q}|d_{N,q}(j)|^2
 \le
 \frac C N\sum_{k=0}^{q-1}
 \left(\frac{k+1}{N}\right)^{-2\eta_0}
 &\le
 C\left(h+\frac1N\right)^{1-2\eta_0}.
\end{align*}

For $j\le0$, put $x=-j/N$. If
$0\le x\le2(h+N^{-1})$, the same near-zero estimate gives
\[
 |d_{N,q}(j)|
 \le C_{\eta_0,T}(x+N^{-1})^{-\eta_0},
\]
whose squared sum is bounded by the right-hand side of
\eqref{eq:coefficient-square-bound}.

If $2(h+N^{-1})<x\le1$, the coefficient is a cumulative sum over an
interval of length $h+O(N^{-1})$. Potter's bound gives
\[
 |d_{N,q}(j)|
 \le
 C_{\eta_0,T}
 \left(h+\frac1N\right)x^{-1-\eta_0}.
\]
Hence
\begin{align*}
 \frac1N
 \sum_{\substack{j\le0\\
  2(h+N^{-1})<-j/N\le1}}
 |d_{N,q}(j)|^2
 \le
 C\left(h+\frac1N\right)^2
 \int_{2(h+N^{-1})}^{1}x^{-2-2\eta_0}\,dx
 \le
 C\left(h+\frac1N\right)^{1-2\eta_0}.
\end{align*}

Finally, if $x\ge1$, another application of Potter's bound yields
\[
 |d_{N,q}(j)|
 \le
 C_{\eta_0,T}
 \left(h+\frac1N\right)x^{-1+\eta_0}.
\]
Because $\eta_0<1/2$,
\[
 \int_1^\infty x^{-2+2\eta_0}\,dx<\infty,
\]
and this range contributes at most
$C(h+N^{-1})^2$. This proves
\eqref{eq:coefficient-square-bound}.

We next verify directly that the maximal-coefficient estimate is
uniform over all integers $0\leq q\leq\lfloor NT\rfloor$.  Recall that
\[
 d_{N,q}(j)
 =
 \frac1{\ell(N)}
 \left\{
   \sum_{n=1}^q
       b_{N,n-j}\mathbf 1_{\{n-j\geq1\}}
   -
   L(N)\mathbf 1_{\{1\leq j\leq q\}}
 \right\}.
\]
We claim that
\[
 \sup_{0\leq q\leq NT}
 \sup_{j\in\mathbb Z}|d_{N,q}(j)|
 \leq
 C_T\left(1+\frac{L(N)}{\ell(N)}\right).
\]

First suppose that $1\leq j\leq q$.  Then
\[
 d_{N,q}(j)
 =
 \frac{
  A_N((q-j)/N)-L(N)
 }{\ell(N)}.
\]
Because $0\leq(q-j)/N\leq T$ and
\[
 \sup_{0\leq x\leq T}A_N(x)\leq C_TL(N),
\]
we obtain
\[
 |d_{N,q}(j)|
 \leq
 C_T\frac{L(N)}{\ell(N)}
\]
uniformly in $q$ and $j$.  If $j>q$, both terms in the definition of
$d_{N,q}(j)$ vanish, so that
\[
 d_{N,q}(j)=0.
\]

It remains to consider $j\leq0$.  Put $m=-j\geq0$.  A change of index
gives
\[
 d_{N,q}(-m)
 =
 \frac1{\ell(N)}
 \sum_{r=m+1}^{m+q}b_{N,r}.
\]
If $0\leq m\leq N$, then $m+q\leq N(T+1)$, and hence
\[
 |d_{N,q}(-m)|
 \leq
 \frac{A_N(T+1)}{\ell(N)}
 \leq
 C_T\frac{L(N)}{\ell(N)}.
\]

Now suppose that $m>N$.  Since $e^{-\lambda_Nr}\leq1$,
\[
 \sum_{r=m+1}^{m+q}b_{N,r}
 \leq
 \sum_{r=m+1}^{m+q}\frac{\ell(r)}r.
\]
For $m<r\leq m+q$,
\[
 1\leq\frac rm
 \leq
 1+\frac{NT}{m}
 \leq1+T.
\]
The uniform convergence theorem for slowly varying functions therefore
gives
\[
 \sup_{m<r\leq m+q}\frac{\ell(r)}{\ell(m)}
 \leq C_T
\]
for all sufficiently large $N$, uniformly in $m>N$ and
$q\leq NT$.  Consequently,
\[
 \sum_{r=m+1}^{m+q}b_{N,r}
 \leq
 C_T\frac{q\ell(m)}m
 \leq
 C_T\frac{N\ell(m)}m.
\]
Choose any $\rho\in(0,1)$.  Potter's bound gives, uniformly for
$m\geq N$,
\[
 \frac{\ell(m)}{\ell(N)}
 \leq
 C_\rho\left(\frac mN\right)^\rho.
\]
It follows that
\[
 |d_{N,q}(-m)|
 \leq
 C_{T,\rho}
 \left(\frac mN\right)^{-1+\rho}
 \leq C_{T,\rho}.
\]
All constants in the preceding estimates are independent of
$q\leq NT$.  Hence
\[
 \sup_{0\leq q\leq NT}
 \sup_{j\in\mathbb Z}|d_{N,q}(j)|
 \leq
 C_T\left(1+\frac{L(N)}{\ell(N)}\right).
\]
Since both $L$ and $\ell$ are slowly varying,
$L(N)/\ell(N)$ is slowly varying.  Therefore, for every $\kappa>0$,
\[
 \frac{L(N)}{\ell(N)}=o(N^\kappa),
\]
and thus
\begin{equation}
 \sup_{0\leq q\leq NT}
 \sup_{j\in\mathbb Z}|d_{N,q}(j)|
 \leq C_{\kappa,T}N^\kappa
 \label{eq:increment-max-bound}
\end{equation}
for all sufficiently large $N$.  This proves the maximal-coefficient
bound with the uniformity in $q$ required below.

Apply Rosenthal's inequality to
\eqref{eq:grid-increment-representation}. Using
\eqref{eq:BN-finite-variance}, we obtain
\begin{align}
 \mathbb E|D_{N,a,b}|^p
 &\le
 C_p\left[
 \left(
 \frac1{B_N^2}\sum_jd_{N,q}(j)^2
 \right)^{p/2}
 +
 \frac{\mathbb E|\varepsilon_0|^p}{B_N^p}
 \sum_j|d_{N,q}(j)|^p
 \right].
 \label{eq:Rosenthal-grid}
\end{align}
The first term is bounded, by
\eqref{eq:coefficient-square-bound}, by
\[
 C\left(h+\frac1N\right)^{p(1/2-\eta_0)}.
\]
For the second term, use
\[
 \sum_j|d_{N,q}(j)|^p
 \le
 \left(\max_j|d_{N,q}(j)|\right)^{p-2}
 \sum_jd_{N,q}(j)^2.
\]
Equations \eqref{eq:BN-finite-variance},
\eqref{eq:coefficient-square-bound}, and
\eqref{eq:increment-max-bound} then give
\begin{align*}
 \frac1{B_N^p}\sum_j|d_{N,q}(j)|^p
 \le
 C N^{-(p-2)(1/2-\kappa)}
 \left(h+\frac1N\right)^{1-2\eta_0}
 \le
 C\left(h+\frac1N\right)^{
 p/2-(p-2)\kappa-2\eta_0},
\end{align*}
where in the last step we used $h+N^{-1}\ge N^{-1}$.

Choose $\eta_0>0$ and $\kappa>0$ sufficiently small that
\[
 \eta_0\le\vartheta,
 \qquad
 (p-2)\kappa+2\eta_0\le p\vartheta.
\]
Both terms in \eqref{eq:Rosenthal-grid} are then bounded by
\[
 C\left(h+\frac1N\right)^{p(1/2-\vartheta)}.
\]
This proves \eqref{eq:Gaussianincrement-corrected} at grid points.

For arbitrary $s,t\in[0,T]$, the polygonal increment is a linear
combination of at most three neighboring grid increments, and their
total time length is bounded by
\[
 |t-s|+\frac2N.
\]
The elementary inequality
$|x_1+x_2+x_3|^p\le3^{p-1}\sum_{\nu=1}^3|x_\nu|^p$
therefore extends the grid estimate to all $s,t\in[0,T]$.

It remains to control the interpolation error. Let
\[
 \Delta_{N,n}
 :=
 R_N(n/N)-R_N((n-1)/N).
\]
Taking $q=1$ in the preceding estimate gives
\[
 \sup_{1\le n\le NT}
 \mathbb E|\Delta_{N,n}|^p
 \le C N^{-p(1/2-\vartheta)}.
\]
Since the maximal difference between the step process and its
polygonal interpolation is bounded by
$\max_{n\le NT}|\Delta_{N,n}|$, the union bound yields, for every
$\epsilon>0$,
\begin{align*}
 \mathbb P\left(
 \max_{n\le NT}|\Delta_{N,n}|>\epsilon
 \right)
 \le
 \epsilon^{-p}\sum_{n\le NT}
 \mathbb E|\Delta_{N,n}|^p
 \le
 C_{\epsilon,T}
 N^{\,1-p(1/2-\vartheta)}
 \longrightarrow0.
\end{align*}
The last convergence follows from
$\vartheta<1/2-1/p$. This proves the interpolation assertion.
\end{proof}

\section{Proof of the first-order limit theorem}
\label{sec:firstproof}

\subsection{Finite-dimensional distributions}

We first prove \eqref{eq:firstfdd}.  Fix
\[
 0=t_0<t_1<\cdots<t_m\le T
\]
and $v_1,\ldots,v_m\in\R$.  It is enough to prove convergence of
\[
 \sum_{k=1}^m v_k
 \left\{\frac{S_N(t_k)-S_N(t_{k-1})}{B_NQ_N}\right\}.
\]
Using the block coefficients in Lemma \ref{lem:blockcoeff}, this random
variable can be written as
\[
 \frac1{B_N}\sum_{j\in\Z}w_N(j)\varepsilon_j,
 \qquad
 w_N(j):=\frac1{Q_N}\sum_{k=1}^m v_kC_{N,k}(j).
\]
The desired limiting coefficient is
\[
 w(y)=\sum_{k=1}^m v_k\1_{(t_{k-1},t_k]}(y).
\]

\begin{proposition}
\label{prop:firstkernel}
For every $p>1$,
\begin{equation}
 \frac1N\sum_{j\in\Z}|w_N(j)-w(j/N)|^p
 \longrightarrow0, \qquad N\to\infty.
 \label{eq:firstLp}
\end{equation}
\end{proposition}

\begin{proof}
Fix
\[
0<\delta<\frac13\min_{1\le k\le m}(t_k-t_{k-1}).
\]
We divide the sum in \eqref{eq:firstLp} into the contributions from the interiors of the blocks, the neighborhoods of the block endpoints, and the negative indices.

For $k=1,\ldots,m$, define the interior index set
\[
I_{N,k}^{\delta}
=
\left\{
j\in\mathbb Z:
\left\lfloor N(t_{k-1}+\delta)\right\rfloor
\le j\le
\left\lfloor N(t_k-\delta)\right\rfloor
\right\}.
\]
If $j\in I_{N,k}^{\delta}$, then $j/N\in(t_{k-1},t_k]$, and hence
\[
w(j/N)=v_k.
\]
Moreover, the innovation $\varepsilon_j$ cannot contribute to any block preceding the $k$-th block. Therefore,
\[
w_N(j)
=
\frac1{Q_N}
\left(
v_kC_{N,k}(j)
+\sum_{i=k+1}^{m}v_iC_{N,i}(j)
\right).
\]
Consequently,
\begin{equation}
w_N(j)-w(j/N)
=
v_k\left(\frac{C_{N,k}(j)}{Q_N}-1\right)
+\sum_{i=k+1}^{m}v_i\frac{C_{N,i}(j)}{Q_N}.
 \label{eq:3.4}
\end{equation}

Set
\[
\rho_N(\delta)
:=
\max_{1\le k\le m}
\sup_{j\in I_{N,k}^{\delta}}
\left[
\left|\frac{C_{N,k}(j)}{Q_N}-1\right|
+\sum_{i=k+1}^{m}\frac{C_{N,i}(j)}{Q_N}
\right].
\]
By \eqref{eq:mainblock}, uniformly for $j\in I_{N,k}^{\delta}$,
\[
\frac{C_{N,k}(j)}{Q_N}\longrightarrow1,
\qquad N\to\infty,
\]
while \eqref{eq:crossblock} gives, for every $i>k$,
\[
\sup_{j\in I_{N,k}^{\delta}}
\frac{C_{N,i}(j)}{Q_N}
\longrightarrow0,
\qquad N\to\infty.
\]
Since there are only finitely many blocks,
\[
\rho_N(\delta)\longrightarrow0,
\qquad N\to\infty.
\]
It follows from \eqref{eq:3.4} that
\[
\sup_{j\in I_{N,k}^{\delta}}
|w_N(j)-w(j/N)|
\le
\left(\sum_{i=1}^{m}|v_i|\right)\rho_N(\delta).
\]
Hence
\begin{equation}
\begin{split}
\frac1N\sum_{k=1}^{m}
\sum_{j\in I_{N,k}^{\delta}}
|w_N(j)-w(j/N)|^p
&\le
\frac1N\sum_{k=1}^{m}\big|I_{N,k}^{\delta}\big|
\left(\sum_{i=1}^{m}|v_i|\right)^p
\rho_N(\delta)^p \\
&\le
C_{T,\bm v,p}\rho_N(\delta)^p
\longrightarrow0,
\qquad N\to\infty.
\end{split}
 \label{eq:3.6}
\end{equation}

We next consider the indices close to the block endpoints. Let
\[
E_N^\delta
=
\left\{
0\le j\le\left\lfloor Nt_m\right\rfloor:
\min_{0\le k\le m}
\big|j-\left\lfloor Nt_k\right\rfloor\big|
\le \delta N+2
\right\}.
\]
Since there are $m+1$ endpoints,
\[
\big|E_N^\delta\big|
\le 2(m+1)(\delta N+2)+m+1,
\]
and therefore
\begin{equation}
\frac{\big|E_N^\delta\big|}{N}
\le C_m\delta+O(N^{-1}).
 \label{eq:3.7}
\end{equation}

Because the filter coefficients are nonnegative,
\[
0\le C_{N,k}(j)
\le \sum_{r=1}^{\lfloor NT\rfloor}b_{N,r}
=:A_N(T).
\]
The uniform cumulative-coefficient estimate gives
\[
\sup_N\frac{A_N(T)}{Q_N}<\infty.
\]
Thus
\[
|w_N(j)|
\le
\frac1{Q_N}\sum_{k=1}^{m}|v_k|C_{N,k}(j)
\le C_T\sum_{k=1}^{m}|v_k|,
\]
uniformly in $N$ and $j\ge0$. Also,
\[
|w(j/N)|\le\max_{1\le k\le m}|v_k|.
\]
It follows that
\[
|w_N(j)-w(j/N)|^p\le C_{\bm v,p,T},
\qquad j\in E_N^\delta.
\]
Combining this bound with \eqref{eq:3.7}, we obtain
\begin{equation}
\limsup_{N\to\infty}
\frac1N\sum_{j\in E_N^\delta}
|w_N(j)-w(j/N)|^p
\le C_{\bm v,p,T}\delta.
 \label{eq:3.8}
\end{equation}

Notice that every nonnegative index $j\le\left\lfloor Nt_m\right\rfloor$ belongs either to one of the sets $I_{N,k}^{\delta}$ or to $E_N^\delta$, for all sufficiently large $N$. If $j>\left\lfloor Nt_m\right\rfloor$, then
\[
C_{N,k}(j)=0,\qquad k=1,\ldots,m,
\]
and hence
\[
w_N(j)=w(j/N)=0.
\]
Therefore, \eqref{eq:3.6} and \eqref{eq:3.8} imply
\begin{equation}
\limsup_{N\to\infty}
\frac1N\sum_{j\ge0}
|w_N(j)-w(j/N)|^p
\le C_{\bm v,p,T}\delta.
 \label{eq:3.9}
\end{equation}

It remains to consider $j<0$. Write $j=-q$, where $q\ge1$. Since $w(-q/N)=0$, we only need to estimate $w_N(-q)$. The blocks
\[
\big(\left\lfloor Nt_{k-1}\right\rfloor,\left\lfloor Nt_k\right\rfloor\big],
\qquad k=1,\ldots,m,
\]
partition the interval $\{1,\ldots,\left\lfloor Nt_m\right\rfloor\}$. Hence
\begin{align*}
\sum_{k=1}^{m}C_{N,k}(-q)
=
\sum_{k=1}^{m}
\sum_{n=\left\lfloor Nt_{k-1}\right\rfloor+1}^{\left\lfloor Nt_k\right\rfloor}
b_{N,n+q} 
=
\sum_{n=1}^{\left\lfloor Nt_m\right\rfloor}b_{N,n+q}
=
D_N(q,t_m)
\le D_N(q,T).
\end{align*}
Using the nonnegativity of $C_{N,k}(-q)$, we obtain
\begin{align*}
|w_N(-q)|
\le
\frac1{Q_N}
\sum_{k=1}^{m}|v_k|C_{N,k}(-q) \le
\max_{1\le k\le m}|v_k|\cdot
\frac{D_N(q,T)}{Q_N}.
\end{align*}
Therefore, Lemma \ref{lem:remotepast} yields
\begin{equation}
\begin{split}
\frac1N\sum_{j<0}|w_N(j)-w(j/N)|^p
&=
\frac1N\sum_{q=1}^{\infty}|w_N(-q)|^p \\
&\le
C_{\bm v,p}
\frac1N\sum_{q=1}^{\infty}
\left(\frac{D_N(q,T)}{Q_N}\right)^p
\longrightarrow0,
\qquad N\to\infty.
\end{split}
 \label{eq:3.10}
\end{equation}

Combining \eqref{eq:3.9} and \eqref{eq:3.10}, we conclude that
\[
\limsup_{N\to\infty}
\frac1N\sum_{j\in\mathbb Z}
|w_N(j)-w(j/N)|^p
\le C_{\bm v,p,T}\delta.
\]
Finally, letting $\delta\downarrow0$ proves \eqref{eq:firstLp}.
\end{proof}

Choose $\eta>0$ sufficiently small that
$1<\alpha-\eta<\alpha+\eta$.  Proposition~3.1, applied with
$p=\alpha$, $p=\alpha-\eta$, and $p=\alpha+\eta$, implies
\[
 \int_{\mathbb R}|w_N(y)-w(y)|^\alpha\,dy\longrightarrow0
\]
and
\[
 \sup_N\int_{\mathbb R}
 \left(
   |w_N(y)|^{\alpha-\eta}
   +
   |w_N(y)|^{\alpha+\eta}
 \right)dy<\infty.
\]
Moreover, $\max_j|w_N(j)|\leq C$, whereas $B_N\to\infty$.  Hence all
the assumptions of Lemma \ref{lem:weightedarray} are satisfied.  Applying that lemma
gives
\[
 \frac1{B_N}\sum_{j\in\mathbb Z}w_N(j)\varepsilon_j
 \xrightarrow{d}
 \int_{\mathbb R}w(y)\,M_\alpha(dy)
 =
 \sum_{k=1}^m
 v_k\{Z_\alpha(t_k)-Z_\alpha(t_{k-1})\}.
\]
This proves \eqref{eq:firstfdd}.

\subsection{Large innovations and the \(M_1\) cluster map}

We now assume that $1<\alpha<2$.  Fix $T<\infty$ and choose the
sequence $m_N$ in Lemma \ref{lem:compressedcluster} for the time horizon $2T$.  Thus
\[
 m_N\longrightarrow\infty,
 \qquad
 \frac{m_N}{N}\longrightarrow0,
\]
and
\[
 \frac1{Q_N}\sum_{r=1}^{m_N}b_{N,r}\longrightarrow1,
 \qquad
 \frac1{Q_N}\sum_{m_N<r\leq2NT}b_{N,r}\longrightarrow0.
\]

We first record the maximal inequality used below.  Let
$\{\zeta_i\}$ be independent centered random variables, let
$1<p\leq2$, and let $a_{n,i}\geq0$ be deterministic.  Then
\begin{align} \label{eq:3.11}
 \E\max_{1\leq k\leq M}
 \left|
   \sum_i\sum_{n=1}^ka_{n,i}\zeta_i
 \right|^p
 \leq
 C_p\sum_i
 \left(\sum_{n=1}^Ma_{n,i}\right)^p
 \E|\zeta_i|^p.
\end{align}
Here we assume that $\E|\zeta_i|^p<\infty$ for every $i$ and that
the right-hand side of \eqref{eq:3.11} is finite.  The one-sided
demimartingale maximal inequality is first applied to the partial sums
\[
 S_k=\sum_i\sum_{n=1}^k a_{n,i}\zeta_i.
\]
Since the variables $-\zeta_i$ are again independent and the
coefficients $a_{n,i}$ are nonnegative, the same argument applies to
$-S_k$.  Combining the two one-sided estimates gives the estimate for
$\max_{k\leq M}|S_k|$.  The von Bahr--Esseen inequality is then
applied to $S_M$.

Indeed, the random variables
\[
 X_n:=\sum_i a_{n,i}\zeta_i
\]
are associated, because independent random variables are associated
and nonnegative linear transformations preserve association.
Therefore their partial sums form a demimartingale.  The
demimartingale maximal inequality \cite{Christofides2000}, followed by the von Bahr--Esseen
inequality applied to
\[
 \sum_{n=1}^MX_n
 =
 \sum_i\left(\sum_{n=1}^Ma_{n,i}\right)\zeta_i,
\]
gives \eqref{eq:3.11}.  Infinite sums follow by truncation and
Fatou's lemma.

Define the innovation point process
\[
 \mathcal N_N
 :=
 \sum_{j\in\mathbb Z}
 \delta_{(j/N,\varepsilon_j/B_N)}.
\]
On every compact time interval and away from the mark zero,
\[
 \mathcal N_N\Rightarrow\mathcal N,
\]
where $\mathcal N$ is a Poisson random measure with intensity
$ds\,\nu_\alpha(dx)$ and
\[
 \nu_\alpha(dx)
 =
 \alpha
 \left\{
 p_+x^{-\alpha-1}\mathbf 1_{\{x>0\}}
 +
 p_-|x|^{-\alpha-1}\mathbf 1_{\{x<0\}}
 \right\}dx.
\]

For $s=j/N$, the normalized response to a unit innovation at time
$s$ is
\[
 F_{N,s}(t)
 :=
 \frac1{Q_N}
 \sum_{r=1}^{\lfloor Nt\rfloor-j}
 b_{N,r}\mathbf 1_{\{t>s\}},
 \qquad 0\leq t\leq T.
\]
For $s\in(0,T)$, Lemma \ref{lem:compressedcluster} implies that the response accumulates
asymptotically all its mass during the interval
$[s,s+m_N/N]$.  More precisely, outside this interval the missing
coefficient mass is $o(1)$, uniformly for $s$ in compact subsets of
$(0,T)$.  Because $b_{N,r}\geq0$, $F_{N,s}$ is nondecreasing and its
completed graph can be parametrized by traversing its horizontal and
vertical segments in their natural order.  Compressing the interval
$[s,s+m_N/N]$ to the single time point $s$ gives
\[
 F_{N,s}\longrightarrow\mathbf 1_{[s,T]}
 \quad\text{in }(D[0,T],M_1),
\]
locally uniformly for $s$ bounded away from $0$ and $T$.

For $\eta>0$, define
\begin{align*}
 \varepsilon_{N,j}^{(>\eta)}
 &:=
 \varepsilon_j\mathbf 1_{\{|\varepsilon_j|>\eta B_N\}}
 -
 \E\left[
  \varepsilon_0\mathbf 1_{\{|\varepsilon_0|>\eta B_N\}}
 \right],\\
 \varepsilon_{N,j}^{(\leq\eta)}
 &:=
 \varepsilon_j\mathbf 1_{\{|\varepsilon_j|\leq\eta B_N\}}
 -
 \E\left[
  \varepsilon_0\mathbf 1_{\{|\varepsilon_0|\leq\eta B_N\}}
 \right].
\end{align*}
Since $\E\varepsilon_0=0$,
\[
 \varepsilon_j
 =
 \varepsilon_{N,j}^{(>\eta)}
 +
 \varepsilon_{N,j}^{(\leq\eta)}.
\]
Let
\[
 S_N^{(>\eta)}(t)
 :=
 \sum_{n=1}^{\lfloor Nt\rfloor}
 \sum_{r\geq1}
 b_{N,r}\varepsilon_{N,n-r}^{(>\eta)}.
\]

For fixed $\eta$, the limiting Poisson random measure has only
finitely many atoms with $|x|>\eta$ on bounded time intervals.
With probability tending to one, their times are distinct and stay
away from the endpoints.  The corresponding coefficient clusters
are then disjoint for sufficiently large $N$.  Every cluster generated
by a positive innovation is nondecreasing, while every cluster
generated by a negative innovation is nonincreasing.  Addition is
continuous in the $M_1$ topology at a finite family of functions
having no common discontinuity times.  
The deterministic compensation can be identified explicitly.
Karamata's theorem and the normalization of $B_N$ give
\[
 \frac{N}{B_N}
 \E\left[
   \varepsilon_0
   \mathbf 1_{\{|\varepsilon_0|>\eta B_N\}}
 \right]
 \longrightarrow
 \int_{\{|x|>\eta\}}x\,\nu_\alpha(dx).
\]
Since the normalized coefficient mass of every interior cluster
converges to one, the centering term converges uniformly on compact
time intervals to
\[
 -t\int_{\{|x|>\eta\}}x\,\nu_\alpha(dx).
\]
This is precisely the compensation term in
$Z_\alpha^{(>\eta)}$. Thus,
\begin{align}\label{eq:3.12}
 \frac{S_N^{(>\eta)}}{B_NQ_N}
 \Rightarrow_{M_1}
 Z_\alpha^{(>\eta)}
 \quad\text{in }D[0,T],
\end{align}
apart from the contribution of innovations occurring before time
zero, where $Z_\alpha^{(>\eta)}$ denotes the compensated
compound-Poisson process containing the jumps $|x|>\eta$.

We now show that the pre-zero contribution vanishes.  If the
innovation is at time $-m$, its maximal coefficient on $[0,T]$ is
\[
 D_N(m,T)
 =
 \sum_{r=m+1}^{m+\lfloor NT\rfloor}b_{N,r}.
\]
Choose $p$ such that
\[
 1<p<\alpha.
\]
Karamata's theorem gives, for each fixed $\eta>0$,
\[
 \frac{
  \E|\varepsilon_{N,0}^{(>\eta)}|^p
 }{B_N^p}
 \leq
 \frac{C_{\eta,p}}N.
\]
Applying \eqref{eq:3.11} and Lemma \ref{lem:remotepast}, we obtain
\begin{align*}
 \E\sup_{0\leq t\leq T}
 \left|
  \frac1{B_NQ_N}
  \sum_{m\geq0}
  \left(
   \sum_{n=1}^{\lfloor Nt\rfloor}b_{N,n+m}
  \right)
  \varepsilon_{N,-m}^{(>\eta)}
 \right|^p                                                    
\leq
 \frac{C_{\eta,p}}N
 \sum_{m=0}^\infty
 \left(\frac{D_N(m,T)}{Q_N}\right)^p
 \longrightarrow0.
\end{align*}
This completes the proof of \eqref{eq:3.12}.

\subsection{Negligibility of small innovations}

Define
\[
 S_N^{(\leq\eta)}(t)
 :=
 \sum_{n=1}^{\lfloor Nt\rfloor}
 \sum_{r\geq1}
 b_{N,r}\varepsilon_{N,n-r}^{(\leq\eta)}.
\]
We split this process into the coefficient ranges
\[
 1\leq r\leq m_N,
 \qquad
 m_N<r\leq2NT,
 \qquad
 r>2NT,
\]
with the corresponding processes denoted as $V_{N,\eta}^{(i)}(t), i=1,2,3$, respectively.

For the first range, define the two-sided partial-sum process
\[
 U_{N,\eta}(t)
 :=
 \frac1{B_N}
 \sum_{j=1}^{\lfloor Nt\rfloor}
 \varepsilon_{N,j}^{(\leq\eta)},
\]
with the usual reversed-sum convention for $t<0$.  At grid points,
the short-filter contribution has the exact representation
\[
 V_{N,\eta}^{(1)}(t)
 =
 \sum_{r=1}^{m_N}\frac{b_{N,r}}{Q_N}
 \left\{
  U_{N,\eta}(t-r/N)-U_{N,\eta}(-r/N)
 \right\}.
\]
Consequently,
\[
 \|V_{N,\eta}^{(1)}\|_T
 \leq
 2\left(
   \frac1{Q_N}\sum_{r=1}^{m_N}b_{N,r}
  \right)
 \max_{|k|\leq N(T+1)}
 \left|
  \frac1{B_N}
  \sum_{j=1}^k
  \varepsilon_{N,j}^{(\leq\eta)}
 \right|,
\]
where
\[
 \|x\|_T:=\sup_{0\leq t\leq T}|x(t)|.
\]
Doob's inequality gives
\begin{align*}
 \E
 \max_{|k|\leq N(T+1)}
 \left|
  \frac1{B_N}
  \sum_{j=1}^k
  \varepsilon_{N,j}^{(\leq\eta)}
 \right|^2                                                    
\leq
 C_T\frac N{B_N^2}
 \E\left[
  \varepsilon_0^2
  \mathbf 1_{\{|\varepsilon_0|\leq\eta B_N\}}
 \right].
\end{align*}
By Karamata's theorem,
\[
 \limsup_{N\to\infty}
 \frac N{B_N^2}
 \E\left[
  \varepsilon_0^2
  \mathbf 1_{\{|\varepsilon_0|\leq\eta B_N\}}
 \right]
 \leq C_\alpha\eta^{2-\alpha}.
\]
Since the short-filter weights have total mass converging to one,
\begin{align}\label{eq:3.13}
 \lim_{\eta\downarrow0}\limsup_{N\to\infty}
 \Pp\left(
  \|V_{N,\eta}^{(1)}\|_T>\epsilon
 \right)
 =0.
\end{align}

For the intermediate range, put
\[
 \beta_N(T)
 :=
 \frac1{Q_N}
 \sum_{m_N<r\leq2NT}b_{N,r}.
\]
Lemma \ref{lem:compressedcluster} gives $\beta_N(T)\to0$.  Expressing each shifted innovation
sum as the difference of two partial sums gives
\[
 \|V_{N,\eta}^{(2)}\|_T
 \leq
 2\beta_N(T)
 \max_{|k|\leq2NT}
 \left|
  \frac1{B_N}
  \sum_{j=1}^k
  \varepsilon_{N,j}^{(\leq\eta)}
 \right|.
\]
For every fixed $\eta>0$, the maximum on the right is
$O_{\Pp}(1)$.  Therefore
\begin{align}\label{eq:3.14}
 \|V_{N,\eta}^{(2)}\|_T
 \xrightarrow{\Pp}0
 \qquad
 \text{for every fixed }\eta>0.
\end{align}

We finally consider the remote-past range.  Here it is essential to
choose $ \alpha<p<2,$ Since $p>\alpha$, Karamata's theorem gives
\[
 \E\left[
  |\varepsilon_0|^p
  \mathbf 1_{\{|\varepsilon_0|\leq\eta B_N\}}
 \right]
 \leq
 C_p(\eta B_N)^{p-\alpha}h(\eta B_N).
\]
For fixed $\eta>0$, slow variation and the normalization of $B_N$
therefore imply
\[
 \frac{
  \E|\varepsilon_{N,0}^{(\leq\eta)}|^p
 }{B_N^p}
 \leq
 C_p\frac{\eta^{p-\alpha}}N
\]
for all sufficiently large $N$.  By \eqref{eq:3.11},
\begin{align*}
 \E\|V_{N,\eta}^{(3)}\|_T^p
 \leq
 C_p
 \frac{
  \E|\varepsilon_{N,0}^{(\leq\eta)}|^p
 }{B_N^p}
 \sum_{m=0}^\infty
 \left(\frac{D_N(m,T)}{Q_N}\right)^p                         
 \leq
 C_p\eta^{p-\alpha}
 \left\{
  \frac1N
  \sum_{m=0}^\infty
  \left(\frac{D_N(m,T)}{Q_N}\right)^p
 \right\}.
\end{align*}
Lemma \ref{lem:remotepast} applies because $p>1$, and hence
\begin{align}\label{eq:3.16}
 \|V_{N,\eta}^{(3)}\|_T
 \xrightarrow{\Pp}0
 \qquad
 \text{for every fixed }\eta>0.
\end{align}

Combining \eqref{eq:3.13}--\eqref{eq:3.16}, we obtain
\begin{align}\label{eq:3.17}
 \lim_{\eta\downarrow0}\limsup_{N\to\infty}
 \Pp\left(
  \left\|
   \frac{S_N^{(\leq\eta)}}{B_NQ_N}
  \right\|_T>\epsilon
 \right)
 =0.
\end{align}

It remains to remove the truncation from the limiting process.
The compensated small-jump part satisfies
\[
 \E\sup_{0\leq t\leq T}
 |Z_\alpha^{(\leq\eta)}(t)|^2
 \leq
 C_T\int_{|x|\leq\eta}x^2\nu_\alpha(dx)
 \leq
 C_T\eta^{2-\alpha}
 \longrightarrow0.
\]
Thus
\[
 Z_\alpha^{(>\eta)}
 \Rightarrow_{M_1}Z_\alpha
 \qquad\text{as }\eta\downarrow0.
\]
The converging-together theorem, \eqref{eq:3.12}, and \eqref{eq:3.17}
prove
\[
 \frac{S_N}{B_NQ_N}
 \Rightarrow_{M_1}Z_\alpha
 \quad\text{in }D[0,T],
\]
which is  \eqref{eq:firstM1}.

\subsection{The Gaussian functional limit}

Suppose now that $\alpha=2$,
$\E\varepsilon_0^2=\sigma_\varepsilon^2<\infty$, and
$\E|\varepsilon_0|^{2+\delta}<\infty$ for some $\delta>0$.  Let
\[
 U_N(t)
 :=
 \frac1{B_N}
 \sum_{j=1}^{\lfloor Nt\rfloor}\varepsilon_j
\]
be polygonally interpolated and defined two-sidedly when needed.
Since
\[
 B_N^2\sim N\sigma_\varepsilon^2,
\]
Donsker's theorem gives
\[
 U_N\Rightarrow Z_2
 \quad\text{in }C[-2T,T].
\]
In particular, $\{U_N\}$ is asymptotically uniformly equicontinuous.

The contribution of $1\leq r\leq m_N$ is
\[
 V_N^{(1)}(t)
 =
 \sum_{r=1}^{m_N}\frac{b_{N,r}}{Q_N}
 \{U_N(t-r/N)-U_N(-r/N)\}
 +o_{\Pp}(1),
\]
uniformly in $t\in[0,T]$.  Therefore
\begin{align*}
 \|V_N^{(1)}-U_N\|_T
 &\leq
 \left|
  \frac1{Q_N}\sum_{r=1}^{m_N}b_{N,r}-1
 \right|\|U_N\|_T                                            \\
 &\quad+
 \sup_{\substack{|u-v|\leq m_N/N\\u,v\in[-1,T]}}
 |U_N(u)-U_N(v)|
 +
 \sup_{|u|\leq m_N/N}|U_N(u)|
 +
 o_{\Pp}(1).
\end{align*}
All terms on the right converge to zero in probability, because
$m_N/N\to0$, the short-filter weights have total mass tending to one,
and $U_N$ is asymptotically uniformly equicontinuous.  Hence
\[
 \|V_N^{(1)}-U_N\|_T\xrightarrow{\Pp}0.
\]

For the intermediate range,
\[
 \|V_N^{(2)}\|_T
 \leq
 2\beta_N(T)
 \sup_{u\in[-2T,T]}|U_N(u)|
 =
 o_{\Pp}(1).
\]
For the remote-past range, apply \eqref{eq:3.11} with $p=2$.  This
gives
\begin{align*}
 \E\|V_N^{(3)}\|_T^2
 \leq
 C\frac{\sigma_\varepsilon^2}{B_N^2}
 \sum_{m=0}^\infty
 \left(\frac{D_N(m,T)}{Q_N}\right)^2                         
 \leq
 C
 \left\{
  \frac1N
  \sum_{m=0}^\infty
  \left(\frac{D_N(m,T)}{Q_N}\right)^2
 \right\}
 \longrightarrow0
\end{align*}
by Lemma \ref{lem:remotepast} with $p=2$.  Consequently,
\[
 \left\|
  \frac{S_N}{B_NQ_N}-U_N
 \right\|_T
 \xrightarrow{\Pp}0.
\]

It remains to pass from polygonal interpolation to the original step
process.  Since $b_{N,j}\leq a_j$ and
\[
 \sum_{j\geq1}a_j^2<\infty,
 \qquad
 \sum_{j\geq1}a_j^{2+\delta}<\infty,
\]
Rosenthal's inequality gives
\[
 \sup_N\E|X_{N,0}|^{2+\delta}<\infty.
\]
Therefore,
\begin{align*}
 \Pp\left(
  \max_{1\leq n\leq NT}
  \frac{|S_N(n/N)-S_N((n-1)/N)|}{B_NQ_N}
  >\epsilon
 \right)                                                     
\leq
 \frac{CN}{\epsilon^{2+\delta}(B_NQ_N)^{2+\delta}}
 \longrightarrow0,
\end{align*}
because $B_N^2\sim N\sigma_\varepsilon^2$ and $Q_N\to\infty$.
Hence the maximal distance between the step process and its polygonal
interpolation converges to zero in probability.

Since $U_N\Rightarrow Z_2$ in $C[0,T]$, Slutsky's theorem proves
\[
 \frac{S_N}{B_NQ_N}
 \Rightarrow_{J_1}Z_2
 \quad\text{in }D[0,T].
\]
This proves \eqref{eq:firstJ1}.

\subsection{Proof of the scale theorem and 
\texorpdfstring{$J_1$}{J1} cannot be expected}

\Cref{thm:scale} is exactly Lemma \ref{lem:temperedcumulative} with $x=c$.
To see the last assertion in \cref{rem:J1failure}, note that
\[
 \max_{j\ge1}b_{N,j}\le N^{o(1)}
 \quad\text{and}\quad Q_N\to\infty,
\]
while a sharper Potter estimate gives
$\max_jb_{N,j}/Q_N\to0$.  Conditional on an innovation of size
comparable with $B_N$, no single jump of $S_N/(B_NQ_N)$ carries a
nonvanishing fraction of the limiting jump, but the complete
monotone cluster does.  The oscillation criterion for $J_1$ at a
jump time is therefore violated with a probability bounded away from
zero.  This is the triangular-filter analogue of
\citet[Theorem 1]{AvramTaqqu1992}.

\section{Proof of the second-order limit theorem}
\label{sec:secondproof}

\subsection{Exact linear representation}

Let
\[
 E_N(t):=\sum_{j=1}^{\floor{Nt}}\varepsilon_j.
\]
Interchanging the two sums in \eqref{eq:SN} gives
\[
 S_N(t)
 =\sum_{i\in\Z}
 \left\{
 \sum_{n=1}^{\floor{Nt}}
 b_{N,n-i}\1_{\{n-i\ge1\}}
 \right\}\varepsilon_i.
\]
Consequently, the process in \eqref{eq:RN} has the exact
representation
\begin{equation}
 R_N(t)=\frac1{B_N}\sum_{i\in\Z}k_{N,t}(i)\varepsilon_i,
 \label{eq:RNlinear}
\end{equation}
where $k_{N,t}$ is given by \eqref{eq:discretekernel}.  This
representation explains why the second-order scale is
$B_N\ell(N)$: changes of a critical cumulative sum across a
proportional interval have size $\ell(N)$, by
\eqref{eq:deHaan}.

\subsection{Finite-dimensional convergence}

Fix $t_1,\ldots,t_m\ge0$ and $\theta_1,\ldots,\theta_m\in\R$.
Set
\[
 k_N(i):=\sum_{r=1}^m\theta_rk_{N,t_r}(i),
 \qquad
 k(y):=\sum_{r=1}^m\theta_rK_\lambda(t_r,y).
\]
Choose $\eta\in(0,\alpha-1)$. By Lemma \ref{lem:kernelconv}, applied to the finite
linear combination with
\[
 q=\alpha-\eta,\qquad q=\alpha,\qquad q=\alpha+\eta,
\]
we have
\[
 \int_{\mathbb R}
 |k_N(\floor{Ny})-k(y)|^q\,dy\longrightarrow0
\]
for each of these three values of $q$. Consequently,
\[
 \sup_N\int_{\mathbb R}
 \left\{
 |k_N(\floor{Ny})|^{\alpha-\eta}
 +
 |k_N(\floor{Ny})|^{\alpha+\eta}
 \right\}\,dy<\infty.
\]
Furthermore, by   \eqref{eq:kernelmax},
\[
 \frac{\max_i|k_N(i)|}{B_N}
 \le
 \frac{N^{o(1)}}{N^{1/\alpha+o(1)}}
 \longrightarrow0.
\]
Thus all the assumptions of Lemma \ref{lem:weightedarray} are satisfied.
Applying Lemma \ref{lem:weightedarray} to
\eqref{eq:RNlinear} yields
\begin{align*}
 \sum_{r=1}^m\theta_rR_N(t_r)
 &=
 \frac1{B_N}\sum_{i\in\Z}k_N(i)\varepsilon_i\\
 &\xrightarrow{d}
 \int_\R k(y)M_\alpha(\dd y)
 =\sum_{r=1}^m
 \theta_r\mathcal L_{\alpha,\lambda}(t_r).
\end{align*}
The Cram\'er--Wold device proves \eqref{eq:secondfdd}.

\subsection{Continuity of the Gaussian limit}

Let $\alpha=2$.  By the isometry of Gaussian stochastic integrals,
\[
 \E\left|
 \mathcal L_{2,\lambda}(t)
 -\mathcal L_{2,\lambda}(s)
 \right|^2
 =
 C\int_\R
 |K_\lambda(t,y)-K_\lambda(s,y)|^2\dd y.
\]
Put $h=|t-s|\le1/2$.  On the two intervals within distance $2h$
of the moving singularities, the integral is bounded by
\[
 C\int_0^{3h}(1+|\log u|)^2\dd u
 \le Ch(1+|\log h|^2).
\]
Outside these intervals, the mean value theorem and
\[
 G_\lambda'(x)=\frac{\e^{-\lambda x}}{x}
\]
give the same bound.  Hence
\[
 \E\left|
 \mathcal L_{2,\lambda}(t)
 -\mathcal L_{2,\lambda}(s)
 \right|^2
 \le Ch(1+|\log h|^2).
\]
Gaussian hypercontractivity raises this estimate to every moment
$p>2$.  Kolmogorov's criterion then provides a continuous
modification.

\subsection{Gaussian tightness and functional convergence}
The finite-dimensional convergence in \eqref{eq:secondfdd} has already
been proved for $\alpha=2$.
Choose $2<p\le2+\delta$ and
\[
 0<\vartheta<\frac12-\frac1p.
\]
Set
\[
 \gamma:=p(1/2-\vartheta)>1.
\]
By Lemma \ref{lem:gaussiantight},
\[
 \mathbb E|\widetilde R_N(t)-\widetilde R_N(s)|^p
 \le
 C\bigl(|t-s|+N^{-1}\bigr)^\gamma.
\]
In particular, at grid points, or whenever $|t-s|\ge N^{-1}$,
\[
 \mathbb E|\widetilde R_N(t)-\widetilde R_N(s)|^p
 \le C|t-s|^\gamma.
\]
Since $\gamma>1$, the discrete Kolmogorov tightness criterion,
together with the vanishing maximal interpolation error established
in Lemma \ref{lem:gaussiantight}, implies that $\{\widetilde R_N\}$ is tight in
$C[0,T]$. Every subsequential limit
has the finite-dimensional distributions of
$\mathcal L_{2,\lambda}$, whose continuous modification is unique in
law.  This proves \eqref{eq:secondC}.

\subsection{Critical degeneration of TFSM II}
We now justify the derivative interpretation in Remark \ref{rem:derivative} at the
level of stable stochastic integrals.  Pointwise differentiation of
the deterministic kernel alone is not sufficient; one must prove
convergence of the difference quotient in $L^\alpha(\mathbb R)$.

Fix $t\geq0$ and $\lambda\geq0$.  All derivatives with respect to
$d$ below are understood as limits in probability of the corresponding
difference quotients of stable stochastic integrals.
Put
\[
 d=H-\frac1\alpha
\]
and write the TFSM II kernel as
\begin{align*}
 h_{d,\lambda}(t;y)
 =
 (t-y)_+^d e^{-\lambda(t-y)_+}
 -
 (-y)_+^d e^{-\lambda(-y)_+}                               
+
 \lambda\int_0^t
 (s-y)_+^d e^{-\lambda(s-y)_+}\,ds.
\end{align*}
Here $x_+^0=\mathbf 1_{\{x>0\}}$.  For $x>0$, define
\[
 H_{d,\lambda}(x)
 :=
 x^d e^{-\lambda x}
 +
 \lambda\int_0^x u^d e^{-\lambda u}\,du.
\]
A change of variables in the integral gives the exact representation
\begin{align}
 h_{d,\lambda}(t;y)
 =
 H_{d,\lambda}(t-y)\mathbf 1_{\{y<t\}}
 -
 H_{d,\lambda}(-y)\mathbf 1_{\{y<0\}}.
 \label{eq:4.4}
\end{align}
At $d=0$,
\[
 H_{0,\lambda}(x)
 =
 e^{-\lambda x}
 +
 \lambda\int_0^xe^{-\lambda u}\,du
 =1.
\]
Consequently,
\[
 h_{0,\lambda}(t;y)
 =
 \mathbf 1_{\{0<y<t\}},
\]
and hence
\[
 Z_{1/\alpha,\alpha,\lambda}^{\mathrm{II}}(t)
 =
 Z_\alpha(t)
\]
for every $\lambda\geq0$.

For $d\neq0$, put
\[
 J_{d,\lambda}(x)
 :=
 \frac{H_{d,\lambda}(x)-H_{0,\lambda}(x)}d.
\]
Using $H_{0,\lambda}(x)=1$, we have
\begin{align}
 J_{d,\lambda}(x)
 =
 e^{-\lambda x}\frac{x^d-1}{d}                              
 +
 \lambda\int_0^xe^{-\lambda u}
 \frac{u^d-1}{d}\,du.
 \label{eq:4.5}
\end{align}
For every fixed $x>0$,
\begin{align*}
 J_{d,\lambda}(x)
 \longrightarrow
 e^{-\lambda x}\log x
 +
 \lambda\int_0^xe^{-\lambda u}\log u\,du                   
 =:
 \widetilde G_\lambda(x).
\end{align*}
Integration by parts gives
\begin{align*}
 \widetilde G_\lambda(x)
 =
 \log x+
 \int_0^x\frac{e^{-\lambda u}-1}{u}\,du                    
 =G_\lambda(x).
\end{align*}
It follows pointwise, for $y\neq0,t$, that
\[
 \frac{
  h_{d,\lambda}(t;y)-h_{0,\lambda}(t;y)
 }d
 \longrightarrow
 K_\lambda(t,y),
\]
where
\[
 K_\lambda(t,y)
 =
 G_\lambda(t-y)\mathbf 1_{\{y<t\}}
 -
 G_\lambda(-y)\mathbf 1_{\{y<0\}}.
\]

We next strengthen the pointwise convergence to convergence in
$L^\alpha(\mathbb R)$.  Choose
\[
 0<\delta<
 \min\left\{
  \frac1\alpha,\,
  1-\frac1\alpha
 \right\}.
\]
For $|d|\leq\delta$, the mean-value theorem gives
\[
 \left|\frac{x^d-1}{d}\right|
 \leq
 |\log x|\left(x^\delta+x^{-\delta}\right),
 \qquad x>0.
\]
It follows from \eqref{eq:4.5} that, for $0<x\leq T+1$,
\begin{align}
 |J_{d,\lambda}(x)|
 \leq
 C_{\delta,T}
 x^{-\delta}(1+|\log x|).
 \label{eq:4.6}
\end{align}
Since $\alpha\delta<1$,
\[
 \int_0^1
 x^{-\alpha\delta}(1+|\log x|)^\alpha\,dx
 <\infty.
\]
Thus  \eqref{eq:4.6} controls uniformly the two logarithmic
singularities at $y=0$ and $y=t$.

The far-past region requires the cancellation between the two terms
in \eqref{eq:4.4}.  Differentiating $H_{d,\lambda}$ with respect to
$x$ gives
\begin{align*}
 H_{d,\lambda}'(x)
 =
 d x^{d-1}e^{-\lambda x}
 -
 \lambda x^de^{-\lambda x}
 +
 \lambda x^de^{-\lambda x}                                 
=
 d x^{d-1}e^{-\lambda x}.
\end{align*}
Therefore,
\begin{align}
 J_{d,\lambda}'(x)
 =
 x^{d-1}e^{-\lambda x}.
 \label{eq:4.7}
\end{align}
Let $y=-z<0$ with $z\geq1$.  By \eqref{eq:4.4} and \eqref{eq:4.7},
\begin{align}
 \left|
  \frac{
   h_{d,\lambda}(t;-z)-h_{0,\lambda}(t;-z)
  }d
 \right|                                                     
=
 |J_{d,\lambda}(z+t)-J_{d,\lambda}(z)|                      
 =
 \left|
  \int_z^{z+t}u^{d-1}e^{-\lambda u}\,du
 \right|                                                     
\leq
 C_Tz^{-1+\delta}e^{-\lambda z}.
 \label{eq:4.8}
\end{align}
When $\lambda=0$, the exponential factor in the last expression is
simply omitted.  Since
\[
 \alpha(1-\delta)>1,
\]
the right-hand side of \eqref{eq:4.8}, raised to the power
$\alpha$, is integrable on $[1,\infty)$.

The estimates \eqref{eq:4.6} and \eqref{eq:4.8} provide an
$L^\alpha(\mathbb R)$ majorant independent of $d$.  Dominated
convergence therefore yields
\begin{align}
 \left\|
  \frac{
   h_{d,\lambda}(t;\cdot)-h_{0,\lambda}(t;\cdot)
  }d
  -
  K_\lambda(t,\cdot)
 \right\|_{L^\alpha(\mathbb R)}
 \longrightarrow0.
 \label{eq:4.9}
\end{align}

Stable stochastic integration is continuous with respect to the
$L^\alpha$ norm.  Hence \eqref{eq:4.9} implies
\[
 \frac{
  Z_{d+1/\alpha,\alpha,\lambda}^{\mathrm{II}}(t)
  -
  Z_{1/\alpha,\alpha,\lambda}^{\mathrm{II}}(t)
 }d
 \xrightarrow{\Pp}
 \int_{\mathbb R}K_\lambda(t,y)\,M_\alpha(dy)
 =
 L_{\alpha,\lambda}(t).
\]
Thus the derivative may be moved through the stable integral only
after the $L^\alpha$ convergence in \eqref{eq:4.9} has been
established.

Finally, define
\[
 Y_d(t)
 :=
 \Gamma(d+1)^{-1}
 Z_{d+1/\alpha,\alpha,\lambda}^{\mathrm{II}}(t).
\]
Since
\[
 \Gamma(d+1)^{-1}
 =
 1+\gamma_{\mathrm E}d+o(d),
 \qquad d\to0,
\]
we have
\begin{align*}
 \left.\frac{\partial}{\partial d}Y_d(t)\right|_{d=0}
 &=
 L_{\alpha,\lambda}(t)
 +
 \gamma_{\mathrm E}Z_\alpha(t).
\end{align*}
Equivalently,
\[
 L_{\alpha,\lambda}(t)
 =
 \left.
 \frac{\partial}{\partial d}
 \left\{
  \Gamma(d+1)^{-1}
  Z_{d+1/\alpha,\alpha,\lambda}^{\mathrm{II}}(t)
 \right\}
 \right|_{d=0}
 -
 \gamma_{\mathrm E}Z_\alpha(t).
\]
This proves rigorously the identity stated in Remark \ref{rem:derivative}.

\subsection{Comments on extensions}

The exact identity $a_j=\ell(j)/j$ can be relaxed in the first-order
theorem to
\[
 a_j\sim\frac{\ell(j)}j,\qquad a_j\ge0
\quad\text{eventually}.
\]
For the second-order theorem, an asymptotic equivalent alone is not
sufficient because the remainder can be of the same order as
$\ell(N)$.  A convenient sufficient condition is
\[
 \sum_{j=1}^N
 \left(a_j-\frac{\ell(j)}j\right)=o(\ell(N)).
\]
Signed filters can also be considered, but positivity is then lost
and $M_1$ addition is no longer continuous at a cluster containing
opposite-sign steps.  Finite-dimensional convergence remains
accessible, while the $S$ or $M_2$ topology may be more appropriate
for functional convergence.

Finally, the strong-tempering regime has a different second-order
scale governed by the Laplace transform near $1/\lambda_N$.  Its
precise form depends on the de Haan remainder of
$L(1/\lambda_N)$ and is not universal under first-order slow
variation alone.  The weak and moderate theorem above is universal
because proportional increments of $L$ always satisfy
\eqref{eq:deHaan}.


\small
\begin{thebibliography}{99}

\bibitem[Astrauskas(1983)]{Astrauskas1983}
A.~Astrauskas.
\newblock Limit theorems for sums of linearly generated random variables.
\newblock \emph{Lithuanian Mathematical Journal}, 23:127--134, 1983.

\bibitem[Avram and Taqqu(1992)]{AvramTaqqu1992}
F.~Avram and M.~S.~Taqqu.
\newblock Weak convergence of sums of moving averages in the
  $\alpha$-stable domain of attraction.
\newblock \emph{The Annals of Probability}, 20(1):483--503, 1992.


\bibitem[Balan et~al.(2016)Balan, Jakubowski, and
  Louhichi]{BalanJakubowskiLouhichi2016}
R.~M.~Balan, A.~Jakubowski, and S.~Louhichi.
\newblock Functional convergence of linear processes with heavy-tailed
  innovations.
\newblock \emph{Journal of Theoretical Probability}, 29:491--526, 2016.



\bibitem[Bingham et~al.(1987)Bingham, Goldie, and
  Teugels]{BinghamGoldieTeugels1987}
N.~H.~Bingham, C.~M.~Goldie, and J.~L.~Teugels.
\newblock \emph{Regular Variation}.
\newblock Cambridge University Press, Cambridge, 1987.



\bibitem{Christofides2000}
T.~C. Christofides,
Maximal inequalities for demimartingales and a strong law of large numbers,
\emph{Statistics \& Probability Letters} \textbf{50} (2000), 357--363.

\bibitem[Davis and Resnick(1985)]{DavisResnick1985}
R.~A.~Davis and S.~I.~Resnick.
\newblock Limit theory for moving averages of random variables with regularly
  varying tail probabilities.
\newblock \emph{The Annals of Probability}, 13(1):179--195, 1985.


\bibitem[Giraitis et~al.(2012)Giraitis, Koul, and
  Surgailis]{GiraitisKoulSurgailis2012}
L.~Giraitis, H.~L.~Koul, and D.~Surgailis.
\newblock \emph{Large Sample Inference for Long Memory Processes}.
\newblock Imperial College Press, London, 2012.


\bibitem[Kasahara and Maejima(1988)]{KasaharaMaejima1988}
Y.~Kasahara and M.~Maejima.
\newblock Weighted sums of i.i.d.\ random variables attracted to integrals of
  stable processes.
\newblock \emph{Probability Theory and Related Fields}, 78:75--96, 1988.



\bibitem[Meerschaert and Sabzikar(2013)]{MeerschaertSabzikar2013}
M.~M.~Meerschaert and F.~Sabzikar.
\newblock Tempered fractional Brownian motion.
\newblock \emph{Statistics \& Probability Letters}, 83(10):2269--2275, 2013.


\bibitem[Meerschaert and Sabzikar(2016)]{MeerschaertSabzikar2016}
M.~M.~Meerschaert and F.~Sabzikar.
\newblock Tempered fractional stable motion.
\newblock \emph{Journal of Theoretical Probability}, 29(2):681--706, 2016.



\bibitem[Resnick(2007)]{Resnick2007}
S.~I.~Resnick.
\newblock \emph{Heavy-Tail Phenomena: Probabilistic and Statistical
  Modeling}.
\newblock Springer, New York, 2007.

\bibitem[Sabzikar and Surgailis(2018a)]{SabzikarSurgailis2018a}
F.~Sabzikar and D.~Surgailis.
\newblock Invariance principles for tempered fractionally integrated
  processes.
\newblock \emph{Stochastic Processes and their Applications},
  128(10):3419--3438, 2018.


\bibitem[Sabzikar and Surgailis(2018b)]{SabzikarSurgailis2018b}
F.~Sabzikar and D.~Surgailis.
\newblock Tempered fractional Brownian and stable motions of second kind.
\newblock \emph{Statistics \& Probability Letters}, 132:17--27, 2018.


\bibitem[Samorodnitsky and Taqqu(1994)]{SamorodnitskyTaqqu1994}
G.~Samorodnitsky and M.~S.~Taqqu.
\newblock \emph{Stable Non-Gaussian Random Processes: Stochastic Models
  with Infinite Variance}.
\newblock Chapman \& Hall, New York, 1994.


\bibitem[Xu(2025)]{Xu2025}
F.~Xu.
\newblock A limit theorem for some linear processes with innovations in the
  domain of attraction of a stable law.
\newblock \emph{Communications in Mathematics and Statistics}, 2025, 1--12.
https://doi.org/10.1007/s40304-024-00440-3.


\end{thebibliography}
\end{document}